\documentclass[12pt]{amsart}
\usepackage{graphicx,psfrag,epsfig, color,float}
\usepackage{amssymb,amsmath,amscd,amsthm}
\usepackage{graphicx,psfrag,epsfig}
\usepackage{mathtools}
\usepackage[colorlinks,pdfdisplaydoctitle,linkcolor=blue,citecolor=red]{hyperref}
\usepackage{graphicx}
\usepackage{bbm}
\usepackage{hyperref}

\usepackage{tikz}

\newtheorem{theorem}{Theorem}[section]

\newtheorem{lemma}[theorem]{Lemma}
\newtheorem{remark}[theorem]{Remark}
\newtheorem{proposition}[theorem]{Proposition}

\newcommand{\thmref}[1]{Theorem~\ref{#1}}

\newcommand{\lemref}[1]{Lemma~\ref{#1}}

\newcommand{\no}{\nonumber}
\newcommand{\be}{\begin{equation}}

	\newcommand{\ee}{\end{equation}}
\newcommand{\bi}{\begin{itemize}}
	\newcommand{\ei}{\end{itemize}}
\newcommand{\br}{\begin{eqnarray}}
	\newcommand{\er}{\end{eqnarray}}

\newcommand{\eps}{\varepsilon}

\newcommand{\expE}{\mathbb{E}}

\newcommand{\commentout}[1]{}

\long\def\metanote#1#2{{\color{#1}\
\ifmmode\hbox\fi{\sffamily\mdseries\upshape [#2]}\ }}

\long\def\JN#1{\metanote{red!70!black}{{\tiny JN} #1}}

\def\stackunder#1#2{\mathrel{\mathop{#2}\limits_{#1}}}

\def\be{\color{black}}
\def\br{\color{red}}

\def\eps{{\varepsilon}}

\def\EXP{\mathbb{E}}

\def\TOR{\mathbb{T}}
\def\cube{\mathcal{Q}}
\def\real{\mathbb{R}}
\def \bbZ{\mathbb{Z}}
\def \ph{\varphi}
\def\naturals{\mathbb{N}}
\def\complex{\mathbb{C}}

\def\bm{\bar{m}}
\def\bv{{\bf v}}

\def\cL{\mathcal{L}}
\def\cK{\mathcal{K}}

\def\cN{\mathcal{N}}

\def\tc{{\tilde c}}

\def\rh{\varrho}

\date{}

\begin{document}

	\title[Branching diffusion in periodic media]{Asymptotic behavior of branching diffusion processes in periodic media}

\author[P. Hebbar]{Pratima Hebbar\textsuperscript{1}}
\address{%
  \textsuperscript{1} Department of Mathematics, 
Duke University, Box 90320,
  Durham, NC 27708.}
\email{pratima.hebbar@duke.edu}

\author[L. Koralov]{Leonid Koralov\textsuperscript{2}}
\address{%
  \textsuperscript{2} Department of Mathematics,
University of Maryland,
College Park, MD 20742-4015}
\email{koralov@math.umd.edu}

\author[J. Nolen]{James Nolen\textsuperscript{3}}
\address{%
  \textsuperscript{3} Department of Mathematics, 
Duke University, Box 90320,
  Durham, NC 27708.}
\email{nolen@math.duke.edu}

	\maketitle
	
	\begin{abstract} We study the asymptotic behavior of branching diffusion processes in periodic media. For a super-critical branching process, we distinguish two types of behavior for the normalized number of particles in a bounded domain, depending on the distance of the domain from the region where the bulk of the particles is located. At distances that grow linearly in time, we observe intermittency (i.e., the $k$-th moment dominates the $k$-th power of the first moment for some $k$), while, at distances that grow sub-linearly in time, we show that all the moments converge.  A key ingredient in our analysis is a sharp estimate of the transition kernel for the branching process, valid up to linear in time distances from the location of the initial particle.
	\end{abstract}
	
	\section{Introduction}
	Consider a collection of particles $Y_1(t), Y_2(t),\dots$ in $\real^d$ that move diffusively and independently according to
	\begin{align}
		dY_k(t) = b(Y_k(t)) \,dt + \sigma(Y_k) \,dW_k(t), \label{YSDEdef}
	\end{align}
	where $W_k$ denote independent Brownian motions in $\real^d$. Each particle independently branches into two particles or is annihilated at rates the depend on its location: a particle at $x \in \real^d$ branches into two particles at rate $\alpha(x) \geq 0$, and is annihilated at rate $\beta(x) \geq 0$. The newly created particles starting at the location of their parent then repeat this process independently of each other. This process is referred to as a $d$-dimensional branching diffusion process. We suppose that the drift $b(x)$, the non-degenerate diffusion matrix $\sigma(x)$, and the rates $\alpha(x)$ and $\beta(x)$ are all Lipschitz continuous and $\bbZ^d$ periodic (and thus bounded). That is, $b(x + k) = b(x)$ for all $x \in \real^d$ and $k \in \mathbb{Z}^d$, and similarly for $\sigma$, $\alpha$ and $\beta$. In addition, we assume that all the matrix $a(x)$ is $C^1(\real^d)$.
	
	The main topic of interest here is the limiting behavior of branching diffusion processes in periodic media in the supercritical regime. Our main goal is to study the distribution of the number of particles in regions whose spatial location depends on time. With probability that tends to one, the entire population is confined to a region that grows linearly in time (see Chapter 7.3 in the book of Freidlin \cite{freidlin1985}). The effective drift of a branching process can be understood heuristically as the speed at which the bulk of the particles is traveling in space. We will give a precise definition of the effective drift later in Section \ref{asympde}. For a bounded region at a fixed location, assuming that the effective drift is zero, the structure of the population is similar to that in the compact setting. See, for example, Engl\"ander, Harris, Kyprianou \cite{Englander2010} and references therein. For a time dependent region inside the linearly growing front, the normalized number of particles converges almost surely (see, for example, Uchiyama \cite{uchiyama1982} in the case of constant coefficients). The nature of this convergence, however, depends on how distant the region is from the location of the initial particle (assuming for simplicity that the effective drift is zero). At linear in time distances, we will show that  intermittency may occur (i.e., the $k$-th moment dominates the $k$-th power of the first moment for some $k$), while, at distances that grow sub-linearly in time, we will prove that all the moments converge. For the case of homogeneous media and for the case of compactly supported branching term, this question has been studied in the work of Koralov \cite{Kor13} as well as Koralov, Molchanov \cite{KorMol}. 
	
	Given a single particle initially at $x \in \real^d$, the transition kernel $u(t,x,y)$ is defined by 
	\[
	\int_{\real^d} u(t,x,y) f(y) \,dy = \EXP_x \left[ \sum_{k} f(Y_k(t)) \right],
	\]
	where $f\in C_b(\real)$  and the sum is over all particles alive at time $t \geq 0$.   The function $(t,y) \mapsto u(t,x,y)$ satisfies 
	\begin{equation} \label{hhrr1}
		\partial_t u = \cL_x u, \quad x,y \in \real^d, \;\; t > 0,
	\end{equation}
	with initial condition
	\[
	u(0,\cdot, y) = \delta_y(\cdot),
	\]
	where $\cL$ is the operator
	\begin{equation} \label{gener0}
		\cL u = \frac{1}{2} \sum\limits_{ij = 1}^{d} a_{ij}(x) \frac{\partial^2 u}{\partial x_i \partial x_j}
		+ \sum\limits_{i = 1}^{d} b_i(x) \frac{\partial u}{\partial x_i} + r(x)u,
	\end{equation}
	$a(x) = \sigma(x) \sigma^*(x)$, and $r(x) = \alpha(x) - \beta (x)$.  The operator $\cL - r(x)$ is the generator of the process \eqref{YSDEdef}.  
The first step in our analysis is a precise asymptotic description of the transition kernel $u(t,x,y)$, valid up to the large deviation scale, that is, for $\|x -y\|  = O(t)$.
	
	There are two main parts in the asymptotic analysis of $u(t,x,y)$.  First, we transform the operator $\cL$ in order to alter the effective drift of the process, while simultaneously turning the branching rate into a constant. Thus, the problem reduces to studying the transition kernel of an altered diffusion process near the diagonal, where $\|x - y\| = O(\sqrt{t})$. The next part is to prove a local limit theorem for the new transformed kernel at this diffusive scale.  
	
	The ingredients we use to obtain the asymptotics of the transition kernel - exponential change of measure, homogenization and local limit theorems for the resulting diffusion process are fairly standard. In spite of this, the precise asymptotics of the transition kernel that holds up to linear in time distances has not been published, as far as we know (in 2007, Agmon gave a talk \cite{Agmon} where this result was announced). Here, we provide a simple probabilistic proof that establishes  uniform asymptotics of the transition kernel for $d$-dimensional second-order parabolic operators with periodic coefficients. The precise asymptotics in the $1$-dimensional case has been obtained previously by Tsuchida in \cite{TT}.
	
	Prior results in this direction, in $d$ dimensions, give estimates of the heat kernel, as opposed to precise asymptotics. The seminal work of Aronson \cite{A68} gives global estimates on the heat kernel, while in \cite{Norris} Norris proves a generalization of Aronson's Gaussian bounds in the case of periodic coefficients and identifies an effective drift of the heat flow. The upper and lower bounds of Norris \cite{Norris} have different constant prefactor in front of the Gaussian term, although the logarithmic asymptotics are sharp. We provide a stronger result that correctly identifies the main term of the asymptotic expansion of the transition kernel, which is precise up to the domain of large deviations (up to distances in space that are linear in time). The asymptotics of Green's function for the corresponding elliptic problem for different values of the spectral parameter has been studied extensively (see, e.g., Murata, Tsuchida \cite{murata2006}, Kuchment, Raich \cite{kuchment2012}).
	
	The asymptotics proved in Section \ref{asympde} plays a crucial role in analyzing the behavior of the branching diffusion process in periodic media, in Section \ref{inter}. The bulk of the particles will be seen to be located around of $\bar{\bv} t$ where $\bar{\bv}$ denotes the effective drift of the process (defined later at \eqref{effectivedrift}). Let $n^y(t,x)$ denote the number of particles  located in a unit $d$-dimensional cube containing $y\in \real^d$, assuming that, initially, there is one particle located at $x\in \real^d$. In Section \ref{inter1}, for a super-critical branching process, we study the asymptotic behavior of $n^y(t,x)$ in the domain of large deviations, that is when $\|y-\bar{\bv} t\| = O(t)$. We observe the effect of intermittency, that is, for each vector $\bv \in \real^d$, $\bv \neq \bar{\bv}$, there exists $k \geq 2$ such that the $k$-th moment of $n^{\bv t}(t,x)$ grows exponentially faster than the $k$-th power of the first moment. This result was first proved in \cite{KorMol} in the case of a super-critical branching diffusion process in $\real^d$ with identity diffusion matrix, zero drift, and a positive constant potential. Here, in contrast to \cite{KorMol}, we do not have explicit expressions for the transition kernel, but only have asymptotic formulas. This makes the analysis of the higher order moments more involved.  
	
	In Section \ref{totnumber}, we define a sequence of periodic functions $f_k(x)$ that serve as limits for the $k$-th moments  of $N(t,x)/\EXP(N(t,x))$, where  $N(t,x)$ denotes the total number of particles in $\real^d$, assuming that, initially, there is one particle located~at~$x~\in~\real^d$. 
	
	In Section \ref{near0}, we again study $n^y(t,x)$, but here we assume that $\|y-\bar{\bv}t\| = o(t)$. That is, we study the distribution of particles near the region where the bulk of the particles is located (i.e, near $\bar{\bv} t$).  In this region, we show that the $k$-th moment of $n^y(t,x)/\EXP(n^y(t,x))$ converges to the periodic function $f_k(x)$ identified in~Section~\ref{totnumber}. 

There have been several other works on different aspects of branching diffusions in periodic media, and the topic is closely related to reaction-diffusion equations with periodic coefficients. After presenting our results more precisely below, we discuss the relation to some of these other works in Section \ref{sec:discussion}.\\
\\
\\

{\bf Acknowledgements:} The work of James Nolen partially funded by grant DMS-1351653 from the US National Science Foundation and the work of Leonid Koralov and Pratima Hebbar was partially funded by grant W911NF1710419 from the Army Research Office.

\vspace{0.2in}

\section{Asymptotics of the transition kernel}\label{asympde}
Given a positive function  $h:\real^d \to \real$ that is sufficiently smooth, the $h$~--~transform of the operator $\cL$ (given in \eqref{gener0}) is defined as 
\[
(\cL_h f)(x) = \frac{1}{h(x)} \cL(h(x)f(x)).
\]
for each real valued $C^2(\real^d)$ function $f$.

For each $t\geq 0$ and $x,y \in \real^d$, the transition kernel $u^h(t,x,y)$ corresponding to $\cL_h$ satisfies:
\begin{align}
u^h(t,x,y) = \frac{1}{h(x)}u(t,x,y) h(y), \label{phtrans}
\end{align}
where $u(t,x,y)$, satisfying \eqref{hhrr1}, is the transition kernel corresponding to $\cL$ (see Theorem 4.1.1 of \cite{Pinsky}). We choose $h$ from among a special family of eigenfunctions of $\cL$ having exponential growth in a given direction. For $\zeta \in \real^d$, let $\ph_\zeta$ be the principal positive periodic eigenfunction of the operator $e^{-\zeta \cdot x} \cL( e^{\zeta \cdot x} \cdot)$. That is $\ph_\zeta$ satisfies
\begin{equation}\label{phieig}
e^{-\zeta \cdot x} \cL( e^{\zeta \cdot x} \ph_\zeta) = \mu(\zeta) \ph_\zeta, 
\end{equation}
with eigenvalue $\mu(\zeta) \in \real$. Let $\ph_\zeta^*$ denote the solution of the adjoint problem, that is,
\[
e^{\zeta \cdot x} \cL^*(e^{-\zeta \cdot x} \ph_\zeta^*) = \mu^*(\zeta) \ph_\zeta^*
\]
where $\mu^*(\zeta)$ is the principal eigenvalue of the adjoint operator, and hence $\mu^*(\zeta) = \mu(\zeta)$. We normalize  $\ph_\zeta$ and $\ph_\zeta^*$ by 
\begin{equation}\label{noreig}
\int_{[0,1)^d} \ph_\zeta(x) \ph_\zeta^*(x)\,dx = 1 = \int_{[0,1)^d} \ph_\zeta^*(x)\,dx.
\end{equation}
Now we define $h_\zeta$ by
\[
h_\zeta(x) = e^{\zeta \cdot x} \ph_\zeta(x), \quad\text{that is,} \quad \cL h_\zeta = \mu(\zeta) h_\zeta.
\]
With this choice of $h = h_\zeta$, (\ref{phtrans}) can be written as
\begin{align}
u(t,x,y)& =  \frac{h_\zeta(x)}{h_\zeta(y)} u^{h_\zeta}(t,x,y) \no \\
& =  e^{- t \left( \zeta \cdot \frac{(y - x)}{t} - \mu(\zeta) \right)} \frac{\ph_\zeta(x)}{\ph_\zeta(y)}  e^{-t\mu(\zeta)}u^{h_\zeta}(t,x,y), \label{pkrep}
\end{align}
Let us define $p^\zeta(t,x,y) := e^{-t\mu(\zeta)}u^{h_\zeta}(t,x,y)$. The function $p^\zeta(t,x,y)$ is the transition kernel for the operator
\begin{align} 
\cK_\zeta w & := (\cL_{h_\zeta} - \mu(\zeta)) w  \label{defopk} \\
& = \frac{1}{e^{\zeta \cdot x} \ph_\zeta(x)} \cL (e^{\zeta \cdot x} \ph_\zeta(x) w(x)) - \mu(\zeta) w \no \\
& = \frac{1}{2} \sum_{ij} a_{ij} w_{x_i x_j} + \sum_{i} \left( b_i + \sum_j a_{ij} ( \zeta_j + \partial_{x_j} \log \ph_\zeta) \right)w_{x_i}\label{kv}.
\end{align}
Compared to $\cL$, this operators $\cK_\zeta$ has an additional periodic drift $a \nabla \log h_\zeta = a \zeta + a \nabla \log \ph_\zeta$, but no branching term $r(x)$. Let $\psi_\zeta$ and $\psi_\zeta^*$ denote the principal eigenfunctions corresponding to the principal eigenvalue (which is equal to zero) of the operator $\cK_\zeta$ and $\cK_\zeta^*$ on the torus, respectively, and suppose that $$\int_{[0,1)^d} \psi_\zeta(x) \psi_\zeta^*(x)\,dx = 1 = \int_{[0,1)^d} \psi_\zeta^*(x) \,dx.$$ It is easy to see that 
\[
\psi_\zeta(x) \equiv 1 ~~~ \text{and}~~~ \psi_\zeta^*(x) \equiv \ph_\zeta^*(x)\ph_\zeta(x).
\]
Now we choose the direction $\zeta\in \real^d$ in an optimal way. Let $\Phi$ denote the Legendre transform of~$\mu(\zeta)$:
\begin{align}
\Phi(c) = \sup_{\zeta \in \real^d} \left( \zeta \cdot c - \mu(\zeta) \right). \label{Phidef}
\end{align}
The properties of $\mu$, from Theorem 2.10 in Chapter 8 of the book of Pinsky \cite{Pinsky}, guarantee that $\Phi \in C^2$ is well-defined.  In particular $\Phi$ is strictly convex.  For each $c \in \real^d$, the supremum in (\ref{Phidef}) is attained at a unique point which will be denoted by $\hat{\zeta} = \hat{\zeta}(c)$, that is
\[
\Phi(c) =\hat \zeta \cdot c - \mu(\hat \zeta).
\]
Thus, $ c = \nabla \mu(\hat \zeta)$. In addition, for each $c\in \real^d$, we have  $\nabla \Phi(c) = \hat \zeta(c)$. Now, given $(t,x,y) \in \real^+\times \real^d \times \real^d$, let 
\begin{equation}\label{valc}
c = c(t,x,y) = \frac{y-x}{t}.
\end{equation}
Corresponding to this $c$, we choose the unique $\hat{\zeta}$ satisfying: 
\begin{equation}\label{valhatv}
c = \nabla \mu(\hat \zeta) ~~~~\text{or equivalently} ~~~~~\nabla \Phi(c) = \hat \zeta.
\end{equation}
Substituting $\zeta =  \hat \zeta(c)$ in to (\ref{pkrep}), we obtain the identity
\begin{align}
u(t,x,y)&  =  e^{- t \Phi(\frac{y-x}{t})} \frac{\ph_{\hat \zeta}(x)}{\ph_{\hat \zeta}(y)} p^{\hat \zeta}(t,x,y), \quad \quad x,y \in \real^d,\;t > 0 \label{pkrep2}
\end{align}
Therefore, to obtain the exact asymptotics of $u(t,x,y)$ in the domain of large deviations, we need to choose $\hat{\zeta}$ appropriately, and provide an exact asymptotics of the transition density $p^{\hat \zeta}(t,x,y)$. The reason for introducing this transformed kernel is that, momentarily assuming $y = y(t) = x+ct$, the effective drift of the process corresponding to $p^{\hat \zeta}(t,x,y)$ is $c$. And therefore, the problem reduces to estimating the density of the transition kernel of the operator $\cK_{\hat{\zeta}}$ at a diffusive scale.  The following proposition, which will be proved later, gives the exact asymptotics of the transition density $p^{\hat \zeta}(t,x,y)$.
\begin{proposition} \label{prop:kasymp}
Fix $L_0 > 0$. For $(t,x,y) \in \real^+\times \real^d \times \real^d$, define $\hat \zeta = \hat{\zeta}(t,x,y) =  \nabla \Phi(\frac{y-x}{t})$. Then
\begin{align}
\lim_{t \to \infty} \sup_{ \|x - y\| \leq t L_0 } \Big\|\frac{1}{ \ph_{\hat \zeta} (y) \ph^*_{\hat \zeta}(y)} \det[ D^2 \Phi(\frac{y-x}{t}) ]^{-1/2} (2 \pi t)^{d/2}  p^{\hat \zeta}(t,x, y) - 1\Big| = 0.\label{kdrift}
\end{align}

\end{proposition}

From Proposition \ref{prop:kasymp}, the following theorem now follows easily, giving the exact asymptotics of $u(t,x,y)$. As we have mentioned, this result was announced in a talk of Agmon \cite{Agmon} in 2007:

\begin{theorem}\label{thmmain1}
Fix $L > 0$.  The following asymptotic relation holds as $t \to \infty$ for all $x,y \in \real^d$ such that $\|y - x\| \leq L t$:
\begin{align}
u(t,x,y)&  =  (2 \pi t)^{-d/2} \det[ D^2 \Phi(\frac{y-x}{t}) ]^{1/2}  e^{- t \Phi(\frac{y-x}{t})} \ph_{\hat \zeta}(x) \ph^*_{\hat \zeta}(y)    \left[ 1 + o_L(1) \right], \label{ldscalebound}
\end{align}
where $\hat \zeta = \hat{\zeta}(t,x,y) = \nabla \Phi(\frac{y-x}{t})$. 
\end{theorem}

\begin{proof}[Proof of Theorem \ref{thmmain1}]
Fix $L >0$. From Proposition \ref{prop:kasymp} and \eqref{pkrep2}, for all $(t,x,y) \in \real^+\times \real^d \times \real^d$ with $\frac{\|y-x\|}{t} \leq L$, we obtain

\begin{align*}
u(t,x,y)   & =  e^{- t \Phi(\frac{ y - x}{t})} \frac{\ph_{\hat \zeta}(x)}{\ph_{\hat \zeta}(y)} p^{\hat \zeta}(t,x,y)\\ 
& = (\sqrt{2\pi t})^{-d}[\det D^2\Phi\big(\frac{y-x}{t}\big)\big]^{1/2}e^{- t \Phi(\frac{ y - x}{t})} \ph_{\hat \zeta}(x)\ph_{\hat{\zeta}}^*(y)(1  +o(1)),
\end{align*}
uniformly for $\|y-x\|\leq Lt$.	This concludes the proof of Theorem \ref{thmmain1}.
\end{proof}

Since $\Phi$ is strictly convex, we define $\bar{\bv} \in \real^d$ to be the unique minimizer of $\Phi$:
\begin{equation}
\Phi(\bar{\bv}) = \min_{v \in \real^d} \Phi(v) = - \mu(0) \label{effectivedrift}
\end{equation}
We call this $\bar{\bv}$ the effective drift of the branching diffusion process.  The logarithmic asymptotics in Theorem \ref{thmmain1} imply that a majority of the particles are located where $|y - x - \bar{\bv} t| = o(t)$. 

The bounds \eqref{ldscalebound} are valid at the large deviation scale, where $\|y - x\|  \leq O(t)$.  The following Aronson-type estimate provides a Gausian bound on the $u$ that holds for all $x,y  \in \real^d$, although it is less precise than \eqref{ldscalebound}.  It is a consequence of Theorem 1.1 from Norris \cite{Norris}:

\begin{lemma}\label{aronlem}
Let $\bar {\bv}$ be the effective drift. There is a constant $c>0$ such that
			\begin{equation}\label{aronsonforf}
				u(t,x,y + \bar {\bv} t) \leq  c t^{-d/2}\exp\left(- t \Phi(\bar {\bv}) - \frac{\|y-x\|^2}{ct}\right),  \quad \quad \forall \;\;x,y \in \real^d,\;\;t > 0.
			\end{equation}
		\end{lemma}

\begin{proof}[Proof of Lemma \ref{aronlem}] 
From \eqref{pkrep} with $\zeta = 0$, we have
\[
u(t,x,y) = e^{t \mu(0)}  \frac{\varphi_0(x)}{\varphi(y)}p^{0}(t,x,y).
\]
where $p^0(t,x,y)$ is the transition kernel for the operator $\mathcal{K}_0$ in \eqref{kv}, having periodic coefficients, but without a potential term. The effective drift for $p^0$ is precisely $\bar{\bv} = \ell(0) = \nabla \mu(0)$. By Theorem 1.1 from Norris \cite{Norris} there exists $C>0$ such that for all $x,y\in \real^d$ and $t > 0$,
		\[
		C^{-1}t^{-d/2}e^{\frac{-C\|y-x\|^2}{t}} \leq p^0(t,x,y + \bar {\bv} t) \leq  C t^{-d/2}e^{\frac{-\|y-x\|^2}{Ct}}.
		\]
(See \cite{Shab19} Lemma 5.3 for an outline of the comparison of the setting in \cite{Norris} to the setting here).
Recall that $\mu(0) = - \Phi(\bar{\bv})$. In terms of $u$, this implies that
		\[
		u(t,x,y + \bar{\bv} t) \leq  \tilde{C} t^{-d/2}\exp\Big( -t \Phi(\bar{\bv})- \frac{\|y -x\|^2}{\tilde C t}\Big).
		\]

\end{proof}

	\section{Asymptotic behavior of a super-critical branching process in periodic media}\label{inter}

In this section, we study the distribution
of the number of particles in regions whose spatial location depends on time.  Throughout this section, we will assume that the branching diffusion process is super-critical, meaning that
\begin{equation}
\Phi(\bar{\bv}) = -\mu(0) < 0, \label{eq:supercrit}
\end{equation}
where $\bar{\bv}$ is the effective drift defined at \eqref{effectivedrift}. In view of Theorem \ref{thmmain1}, this condition implies that the total mass $\int u(t,x,y)\,dy$ grows exponentially fast, as $t \to \infty$. 

Recall that $n^y(t,x)$ denotes the number of particles  located in a unit $d$-dimensional cube containing $y\in \real^d$, assuming that, initially, there is one particle located at $x\in \real^d$.
We state three theorems that describe different behaviors of the distribution of $n^y(t,x)/\EXP(n^y(t,x))$. The main theorem in this section (\thmref{thm:nof}) shows intermittency (i.e., the $k$-th moment dominates the $k$-th power of the first moment for some k), at locations with linear in time distances from the origin (recall that the bulk of the particles is located at the origin).

\subsection{Intermittency in the domain of large deviations} \label{inter1}

		For $y = (y_1, y_2,\cdots, y_d)$, let $\cube_y^d$ denote the $d$-dimensional cube:
		\[
		\cube_y^d = y + [0,1)^d =  [y_1,y_1+1)\times [y_2, y_2+1) \times \cdots \times [y_d, y_d +1).
		\]
	Recall that $n^{y}(t,x)$  denotes the number of particles
located in  $\cube_y^d$, assuming that, initially, there is one particle located at $x \in \real^d$.
		\begin{theorem}\label{thm:nof}
			For each $k \in \naturals\cup \{0\}$, and each $x \in [0,1)^d$, the following statements hold:
			\begin{itemize}
				\item[(a)] For each $\bv \in \real^d$, there exists the limit,
				\begin{equation}
				\gamma_k(\bv) =	\lim\limits_{t \to \infty} \frac{\ln\EXP(n^{t{\bv}}(t,x)^k)}{t} \in \mathbb{R}. \label{gammakdef}
				\end{equation}
For $k = 1$, $\gamma_1(\bv) = - \Phi(\bv)$. For $k \geq 2$, 
		\begin{equation}
		\gamma_k(\bv) = \sup\limits_{w \in \real^d,\\ u \in (0,1)} \Big[ u \gamma_{k-1}\Big(\frac{\bv - w}{u}\Big) +  u \gamma_{1}\Big(\frac{\bv - w}{u}
		\Big) + (1-u)\gamma_1\Big(\frac{ w}{1-u}\Big)\Big]. \label{gammakdef2}
		\end{equation}
				\item[(b)] Define $G_{k} = \{\bv \in \real^d : \gamma_1(\bv) \geq 0, \gamma_k(\bv) = k	\gamma_1(\bv)\}$ for each $k \in \naturals$. Then sets $\{G_k\}_{k \geq 1}$ are closed subsets of $\real^d$ and $G_{k+1} \subseteq G_k$ for all $k \in \naturals$. There exists a sequence of  constants $\alpha_k > 0$ such that $B_{\alpha_k}(\bar{\bv}) \subseteq G_k$, and $\cap_{k \in \naturals} G_k = \{ \bar{\bv}\}$.
			\end{itemize}
		\end{theorem}

Jensen's inequality implies that $\EXP(n^{t{\bv}}(t,x)^k) \geq (\EXP(n^{t{\bv}}(t,x)))^k$ for each $k \in \naturals$ and $\bv \in \real^d$. Therefore, as long as the limits \eqref{gammakdef} exists, we have, $\gamma_k(\bv) \geq k	\gamma_1(\bv)$ for each $k \in \naturals$. Thus, $G_1 \setminus G_k = \{\bv \in G_1 : \gamma_k(\bv) > k	\gamma_1(\bv)\}$. Notice that Part (b) of Theorem \ref{thm:nof} implies that $G_1 \setminus G_{k^*}$ is non-empty for some $k^* \geq 2$. Thus, for $\bv \in G_1 \setminus G_{k^*}$, 
\[
\lim\limits_{t \to \infty}\frac{\ln(\EXP(n^{t{\bv}}(t,x)^k)))}{t} = \gamma_k(\bv) > k\gamma_1(\bv) =  \lim\limits_{t \to \infty}\frac{ \ln(\EXP(n^{t{\bv}}(t,x)))^k)}{t} .
\]
This is the phenomenon of intermittency. This behavior is markedly different from the behavior in the case when the branching rate $c(x)$ is compactly supported in space. In fact, for super-critical branching processes with compactly supported branching rates, in \cite{KorMol}, it is shown that, that $n^{t \bv}(t,x)$ converges after appropriate scaling as $t\to \infty$, and the quantities $\EXP(n^{t \bv}(t,x)^k)/\EXP(n^{t \bv}(t,x))^k$ converge to the corresponding quantities for the limiting random variable. 
\begin{remark}
Formula \eqref{gammakdef2} essentially provides a criterion for establishing weather intermittency occurs or not, in terms of a variational problem. To see this, we demonstrate the case $k = 2$. If $(w, u) = (\bv(1-u), u)$, and $u\uparrow 1$,  then, the term inside the supremum achieves the value $2\gamma_1(\bv)$. This value of $(w,u)$ lies on the boundary of the domain $\real^d\times (0,1)$. Thus, intermittency would occur if there exists a different pair $(w,u) \in \real^d\times (0,1)$ such that the value of the supremum is greater that $2\gamma_1(\bv)$. Otherwise, intermittency can not occur. 
\end{remark}

				\subsection{Distribution of total number of particles}\label{totnumber}
 Following notation introduced in Section \ref{asympde}, recall that $\ph_0$ is the principal periodic eigenfunction of the operator $\cL$. It satisfies
		\begin{align}
			\cL(\ph_0) = \mu(0) \ph_0, 
		\end{align}
		with eigenvalue $\mu(0) \in \real$. The function $\ph_0^*$ will denote the solution of the adjoint eigenvalue problem:
		\[
		\cL^*(\ph_0^*) = \mu^*(0) \ph_0^*,
		\]
		where $\mu^*(0)$ is the principal eigenvalue of the adjoint operator, and hence $\mu^*(0) = \mu(0)$. We normalize  $\ph_0$ and $\ph_0^*$ by 
		\begin{equation}
			\int_{[0,1)^d} \ph_0(y) \ph_0^*(y)\,dy = 1 = \int_{[0,1)^d} \ph_0^*(y) \,dy.
		\end{equation}
		In this section, to simplify notation, we will denote $\ph_0,\ph^*_0$ and $\mu(0)$ by $\ph, \ph^*$ and $\mu$.
		For $t>0$, $x,y \in [0,1)^d$, let $\rh(t,x,y)$ denote the fundamental solution of the following PDE on the torus:
		\[
		\partial_t\rh(t,x,y) = \cL_x\rh(t,x,y), ~~~~~~~~~~~ \rh(0,x,y) = \delta_y(x).
		\]
		Observe that $C_0, \eps>0$ such that, for every $t>0$,  
		\begin{equation}\label{rhobound}
			\int_{[0,1)^d} \rh(t,x,z) dz \leq C_0e^{t \mu }
		\end{equation}
		Let $N(t,x)$ denote the total number of particles in $\real^d$ at time $t$, assuming that, at time $t = 0$, there is one particle at $x \in [0,1)^d$. In the following theorem, all the moments of the normalized total number of particles are shown to converge.

		\begin{theorem}\label{tot}
			For each $k\in \naturals$, the following limit exists uniformly in $x \in [0,1)^d$:
			\begin{equation}
				\lim\limits_{t \to \infty}\frac{\EXP(N(t,x)^k)}{e^{k\mu t}} = f_k(x),
			\end{equation}
			where the functions $f_k$ are defined recursively as follows,
			\[
			f_1(x) = \ph(x),
			\]
			and, for $k \geq 2$,
			\begin{equation}\label{fk}
				f_k(x) = \sum_{i =1}^{k-1}\beta_i^k \int_0^\infty\int_{[0,1)^d}e^{-k\mu t}\alpha(z)f_i(z)f_{k-i}(z)\rh(t,x,z)dzdt,
			\end{equation}
			where $\beta_i^k = k!/(i!(k-i)!)$. In addition, there exists a real valued random variable $\xi_x$, whose distribution is determined uniquely, such that $\EXP(\xi_x^k) = f_k(x)$ for each $k \in \naturals$.
		\end{theorem}
		
			The functions $f_k(x)$ defined recursively by the formulas \eqref{fk} will be shown to be well defined, that is, the integrals in \eqref{fk} will be shown to be convergent.

			The above theorem implies that the total number of particles $N(t,x)$, normalized by its expected value behaves ``regularly". That is, the $k$-th moment of $N(t,x)$ is commensurate with the $k$-th power of the first moment. In the next section, we show that $n^{y}(t,x)$ also exhibits the same ``regular" behavior when $\|y - t\bar{\bv}\| = o(t)$.  In contrast, in Section \ref{inter} we have shown that $n^{t\bv}(t,x)$ exhibits intermittent behavior when $\bv\neq \bar{\bv}$, i.e., the $k$-th moment  of $n^{t\bv}(t,x)$ grows much faster than the $k$-th power of the first moment for some $k \in \naturals$. 

		\subsection{Distribution of the number of particles near the region where the bulk of the particles is located}\label{near0}
		We show that, at distances that grow sub-linearly in time from the bulk of the particles, all the moments converge. Let $n^{y}(t,x)$ denote the number of particles in $\cube_y^d$ at time $t \in \real^+$, given that there was one particle at $x \in [0,1)^d$ at time $t =0$. Define 
		\[
		g(t,y) = (\sqrt{2\pi t})^{-d}\det[ D^2 \Phi(\bar{\bv}) ]^{1/2}e^{-t\Phi(\frac{y}{t}+ \bar{\bv})}.
		\]
		From this formula for $g(t,y)$, since the minimum of the twice continuously differentiable function $\Phi(\bv)$ is achieved at $\bv = \bar{\bv}$, for each $\alpha\in (0,1)$, we get 
		\begin{equation}\label{inc}
			g(t,\alpha y) \geq g(t,y).
		\end{equation}

		\begin{theorem}\label{allmoments}Let $r(t) = o(t)$ as $t \to \infty$. For each $k \in \naturals$, 
			\[
			\lim\limits_{t \to \infty}\frac{\EXP(n^{y(t)+ \bar{\bv}t}(t,x)^k)}{g(t,y(t))^k} = f_k(x)
			\]
			uniformly in $x \in [0,1)^d$ and $\|y(t)\| \leq r(t)$.
		\end{theorem}

\subsection{Discussion} \label{sec:discussion}


There have been several other works on different aspects of branching diffusions in periodic media, and the topic is closely related to reaction-diffusion equations with periodic coefficients.  In particular, many authors have studied the spreading of wave fronts for reaction diffusion equations with periodic coefficents, having of the general form
\begin{equation} \label{RD}
		\partial_t w = \frac{1}{2} \sum\limits_{ij = 1}^{d} a_{ij}(y) \frac{\partial^2 w}{\partial y_i \partial y_j}
		+ \sum\limits_{i = 1}^{d} b_i(y) \frac{\partial w}{\partial y_i}  + f(y, w)\quad y \in \real^d, \;\; t > 0,
\end{equation}
where $f(y,w)$ is of KPP type, for example $f(y,w) = g(y) w(1 - w)$ with $g(y)$ being periodic, or $f(y,w) = w(g(y) - w)$. See \cite{GF79,freidlin1985,Xin00, BH02, BHR1, BHR2} and references therein. In one space dimension, the distribution of the maximal particle in the branching process, (the particle with largest spatial coordinate) can be expressed in terms of the solution to a reaction-diffusion equation of this KPP-type (see for example \cite{McK75}), so that the asymptotic behavior of wave fronts as $t \to \infty$ gives information about the behavior of the extremal particle in the branching process. A similar interpretation holds in the higher-dimensional setting. When $f(y,w) = g(y)w(1 - w)$ and $g > 0$ is strictly positive, a spreading phenomenon occurs:
	\begin{equation} \lim\limits_{t \to \infty} w(t, \bv t) = \begin{cases}
			0 & \text{for $\bv \in \real^d$ with  $\Phi(\bv) >0$} \\
			1 &  \text{for $\bv \in \real^d$ with  $\Phi(\bv) <0$}. \\
		\end{cases}\end{equation}
	(See Chapter 7 of \cite{freidlin1985}). Hence, the set $\{t \bv \in \real^d \;|\; \Phi(\bv) = 0\}$ is understood as the asymptotic front of the wave as $t \to \infty$. This front matches exactly the set $t \partial G_1$, where $G_1$ is defined in \thmref{thm:nof}. The condition that $g(y) > 0$ is not necessary for such a spreading phenomenon. Berestycki, Hamel and Roques \cite{BHR1}, \cite{BHR2} proved a necessary and sufficient condition for the spreading phenomenon (long-time survival of the branching process), which corresponds to the super-critical condition \eqref{eq:supercrit}.  They also analyze the effect of heterogeneity on the principle eigenvalue $\mu(0)$ of the associated linearized problem, and provide  conditions under which the super-critical condition holds (see Theorem 2.12 of \cite{BHR1}). 

Refinements of the linear spreading rate have been obtained, even in the case of periodic media. For example, Hamel, Nolen, Roquejoffre, and Ryzhik \cite{HNRR16} give sharper asymptotics for such fronts in periodic media in one space dimension, extending to the periodic case a well-known result of Bramson \cite{bramson1978} which shows that the front (median of the extremal particle) moves as $c_1 t - c_2 \log(t) + O(1)$ as $t \to \infty$.  A key part of the analysis in \cite{HNRR16} involves an estimate for a heat kernel analogous to $p^\zeta(t,x,y)$ (for the transformed operator $\cK_\zeta$ at \eqref{kv}), except with Dirichlet boundary condition.  This result was extended to fronts in multiple dimensions by Shabani \cite{Shab19}.   Lubetzky, Thornett, and Zeitouni \cite{LTZ18} have proved related asymptotics for the distribution of the extremal particle of a branching diffusion in periodic media. Unlike these works mentioned above, Theorem \ref{thm:nof} pertains to the structure of the branching process behind the front, where the population is growing.

\section{Proof of Proposition \ref{prop:kasymp}}

Let $X_t$ be the diffusion process with generator $\cK_\zeta$ (defined in \eqref{defopk}),
\begin{equation}\label{yt}
dX_t = V(X_t) dt + \sigma(X_t) \,dW_t,~~~ X_0 = x,
\end{equation}
with 
\[
V_i(x) =  b_i(x) + \sum_j a_{ij}(x) ( \zeta_j + \partial_{x_j} \log \ph_\zeta(x)).
\]
From homogenization theory (see Freidlin \cite{freidlin1964} and the books of Bensoussan, Lions, and Papanicolaou \cite{BLP78} and of Jikov, Kozlov, Oleinik \cite{ZKO}), it is well known that the following result holds for diffusion processes with periodic coefficients: There exists a vector $\ell(\zeta)\in \real^d$ (called the effective drift of $X_t$) and a positive definite matrix $\Xi_\zeta$ (called the effective diffusivity of $X_t$) such that 
\[
\frac{X_t - \ell(\zeta) t}{\sqrt{t}} \to \cN(0,\Xi_\zeta) ~~~\text{as}~~~~ t \to \infty,
\]
in distribution, where $\cN(0,\Xi_\zeta)$ denotes the normal random vector with mean zero and covariance matrix $\Xi_\zeta$. These quantities are given by the formulas:
\begin{equation}\label{effdr}
\ell(\zeta) = \int_{[0,1)^d} V(y) \psi^*_\zeta(y)dy = \int_{[0,1)^d} V(y) \ph_\zeta(y)\ph^*_{\zeta}(y) \,dy,
\end{equation}
\begin{equation}\label{effdiff}
\Xi_\zeta = \int_{[0,1)^d} (\nabla \eta_\zeta + I) a(y) ( \nabla \eta_\zeta + I) \ph_\zeta(y)\ph^*_{\zeta}(y)\,dy,
\end{equation}
where $\eta_\zeta(y)$ is a periodic (vector-valued) solution to
\[
\cK_\zeta \eta_\zeta = \ell(\zeta) - V(y),
\]
which is determined uniquely up to an additive constant. These $\ell(\zeta)$ and $\Xi_\zeta$ are often called the effective drift and the effective diffusivity of the operator $\cK_\zeta$ and hence, of the operator $\cL_{h_\zeta}$ since it only differs from $\cK_\zeta$ by a constant potential term. For the operator $\cL$, notice that effective drift $\bar{\bf v}$, as defined at \eqref{effectivedrift}, corresponds to $\bar{\bf v} = \ell(0)$.


We now state the following lemma about properties of the principal eigenvalue $\mu(\zeta)$. The proof of this lemma can be found in the book of Pinsky \cite{Pinsky} (Chapter 8, Theorem 2.10).


\begin{lemma}\label{lambda}
The function $\mu:\real^d \to \real$ is twice continuously differentiable and strictly convex. In addition, for each $\zeta \in \real^d$,
\begin{equation}\label{222}
\nabla\mu(\zeta) = \ell(\zeta), 
\end{equation}
and, 
\begin{equation}\label{d2mu}
D^2\mu(\zeta) = \Xi_\zeta.
\end{equation}
\end{lemma}
\begin{remark}
From equation \eqref{valc} and \eqref{valhatv} we observe that corresponding to each $(t,x,y) \in \real^+\times \real^d \times \real^d$, we choose $\hat \zeta\in \real^d$ such that 
\begin{equation}\label{vallv}
    \nabla\mu(\hat \zeta) = \ell(\hat \zeta) = \frac{y-x}{t}, ~~~~ \text{or equivalently} ~~~~\nabla \Phi(\frac{y-x}{t}) = \hat \zeta.
\end{equation}
\end{remark}

\noindent Since $\Phi$ is the Legendre transform of the function $\mu$, we have the relation 
\[
D^2\mu(\zeta) = \Big[D^2\Phi(\nabla\mu(\zeta))\Big]^{-1}.
\] 
Therefore, for each $\zeta \in \real^d$, \begin{equation}\label{hessian}
[\det D^2\Phi\big(\ell(\zeta)\big)\big]^{-1/2} = 	\big[\det(\Xi_\zeta)\big]^{1/2}.
\end{equation}

The proof of Proposition \ref{prop:kasymp} is based on estimates of the local averages
\begin{align}
\int_{[0,1)^d} p^\zeta(t,x,z + r) f(r) \,dr,  \label{localavg}
\end{align}
for $z \in \mathbb{Z}^d$ and for appropriate choice of test functions $f$.  We will choose $f \in \mathcal{B}$, where $\mathcal{B}$ is the Banach space of $\bbZ^d$ periodic continuous functions $f: \real^d \to \complex$, equipped with the supremum norm. Observe that \eqref{localavg} has the form
\[
\int_{\mathbb{R}^d} p^\zeta(t,x,y) f(y) g([y] - z) \,dy = \expE^\zeta_x[f(X_t) g([X_t] - z)]
\]
with $g(k) =  {\bf{1}}_{0}(k)$.   For parameters $\chi = (z,x,f,\zeta) \in (\bbZ^d,  \real^d, \mathcal{B}, \real^d)$, we define a family of measures on $\bbZ^d$: 
\begin{equation}
m_t^{\chi}({k}) =  \det(\Xi_\zeta)^{1/2}(\sqrt{2\pi t})^d \EXP_x^\zeta(f(X_t) {\bf{1}}_{\{k\}}([X_t] - z )), \quad k \in \mathbb{Z}^d. \label{mchidef}
\end{equation}
For $g: \bbZ^d \to \real$ having bounded support, we denote the action of $m^\chi_t$ on $g$ by
\begin{equation}
m_t^{\chi}(g) =  \det(\Xi_\zeta)^{1/2}(\sqrt{2\pi t})^d \EXP_x^\zeta(f(X_t) g([X_t] - z )). \label{mchidef2}
\end{equation}
Let
\[
\bar m_t^\chi =  e^{-{\frac{(z - \ell(\zeta)t -[x])^T \Xi_\zeta^{-1} (z - \ell(\zeta)t -[x])}{2t}}} \langle \ph_\zeta\ph_\zeta^*, f \rangle,
\]
which we also regard as a (constant) measure on $\mathbb{Z}^d$: $\bar m_t^\chi(g) = \bar m_t^\chi \sum_{k} g(k)$.  Let $\mathcal B_{+,r}$ be defined by:
\[\mathcal B_{+,r} = \{f: f \in \mathcal{B}, f\geq 0, \|f\|_{\infty} < r\}.\]
Let $B_0(L)= \{\zeta \in \real^d \,|\, |\zeta| \leq L\}$ denote the ball of radius $L$ centered at $0$ in $\real^d$.

\begin{lemma}\label{lem222}
	Let $g:\bbZ^d \to \real$ be any function with bounded support and  $\chi = (z,x,f,\zeta) \in (\bbZ^d, \real^d, \mathcal B_{+,r}, B_0(L))$. Then 
	\[
	\lim\limits_{t \to \infty}\sup\limits_{\chi} |m_t^{\chi}(g) - \bar m_t^{\chi}(g) | = 0.
	\]
\end{lemma}

Before proving this, let us use this to finish the proof of Proposition \ref{prop:kasymp}.

\begin{proof}[Proof of Proposition \ref{prop:kasymp}]

\noindent From Lemma \ref{lem222} above, for the function $g(k) = {\bf{1}}_{0}(k)$, we have
\[
\lim\limits_{t \to \infty}\sup\limits_{\substack{f \in C([0,1)^d)\\ \|f\| < r}} \sup\limits_{\substack{{x \in \real^d,  z \in \bbZ^d}\\ \|\zeta\| \leq L }}		\Big|\det(\Xi_\zeta)^{1/2}(\sqrt{2\pi t})^d  \int_{[0,1)^d}p^\zeta(t,x ,z+w)f(w)dw \]	\begin{equation}\label{result333}-  \exp\Big(-{\frac{(z-[x]-\ell(\zeta)t)^T \Xi_\zeta^{-1}(z-[x]-\ell(\zeta)t)}{2t}}\Big) \langle \ph_\zeta\ph_\zeta^*, f \rangle\Big|= 0.
\end{equation}
		To prove Proposition \ref{prop:kasymp}, we would like to be able to replace $f$ by a delta function at $w \in [0,1)^d$. This is easily justified if we have an appropriate bound on the derivative of $p^\zeta(t,x,y)$ in the $y$ variable. In this case, the weighted average of $p^\zeta$ over a small domain approximates the value of $p^\zeta$ at any point inside the domain. To get such bounds on the derivative of $p^\zeta$, we observe that  $p^\zeta(t,x,y) \leq c/t^{d/2}$ for all $x,y \in \real^d$, since $p^\zeta(t,x,y)$ is the fundamental solution of the PDE with periodic coefficients, with no potential term (see, for example, arguments in the proof of  Lemma \ref{aronlem}). From the Schauder estimate (see, Friedman \cite{Friedman}),  it then follows that, $\|\nabla_yp^\zeta(t,x,y)\| \leq \sup\{ p^\zeta(s,x',y')\Big| s\in (t-1,t), x',y' \in \real^d\} \leq c/(t-1)^{d/2} \leq \tilde c/t^{d/2}$. This is enough to conclude from that 
		\begin{equation}
			\lim\limits_{t \to \infty}\sup\limits_{\substack{{x \in \real^d,w \in [0,1)^d, z \in \bbZ^d}\\ \|\zeta\| \leq L }}	\Big|\det(\Xi_\zeta)^{1/2}(\sqrt{2\pi t})^d  p^\zeta(t,x ,z+w)	
		\end{equation}
		\[
		-  \exp\Big(-{\frac{(z-x-\ell(\zeta)t)^T \Xi_\zeta^{-1}(z-x-\ell(\zeta)t)}{2t}}\Big) \ph_\zeta(y)\ph_\zeta^*(y)\Big|= 0.
		\]
		Writing $y \in \real^d$ instead of $z + w$ with $w \in [0,1)^d, z \in \bbZ^d$, we obtain
		\begin{equation}
		\lim\limits_{t \to \infty}\sup\limits_{\substack{{x,y \in \real^d}\\ \|\zeta\| \leq L }}	\Big|\det(\Xi_\zeta)^{1/2}(\sqrt{2\pi t})^d  p^\zeta(t,x ,y) -  e^{-\frac{t}{2}\big(\frac{y-x}{t}- \ell(\zeta)\big)^T \Xi_\zeta^{-1}\big(\frac{y-x}{t}- \ell(\zeta)\big)} \ph_\zeta(y)\ph_\zeta^*(y) \Big|= 0. \label{vunif}
		\end{equation}
		Note that the exponent in the above formula is slightly different. But the difference is negligible in the limit.
		
		Now suppose that $L_0 > 0$ is fixed, and $\|y-x\|/t \leq L_0$, for all $x,y \in \real^d$ and $t > 0$. Then, recall from \eqref{vallv} that if we choose $c = (y-x)/t$, we have a corresponding $\hat{\zeta}$ such that $\ell(\hat{\zeta}) = c$, or equivalently,  $\nabla\Phi(c) = \hat{\zeta}$.  Morevoer, there is $L$, depending on $L_0$, such that $|\hat{\zeta}| \leq L$ holds if $\|y-x\|/t \leq L_0$. Thus, \eqref{vunif} can be applied to those $c$ and $\hat{\zeta}$ uniformly to obtain
		\[
		\lim\limits_{t \to \infty}\sup\limits_{\|x-y\| \leq tL_0}		\Big|\det(\Xi_\zeta)^{1/2}(\sqrt{2\pi t})^d  p^{\hat{\zeta}}(t,x ,y) -  \ph_{\hat{\zeta}}(y)\ph^*_{\hat{\zeta}}(y)\Big|= 0.
		\]
We claim that for any $L < \infty$ fixed, the periodic eigenfunctions normalized by \eqref{noreig} satisfy
\begin{equation}
\inf_{y \in [0,1]^d} \inf_{\| \zeta \| \leq L} \ph_{\zeta}(y) \ph^*_{\zeta}(y)  > 0. \label{philower}
\end{equation}
Therefore, substituting $[\det D^2\Phi\big(\frac{y-x}{t}\big)\big]^{-1/2} = 	\big[\det(\Xi_{\hat{\zeta}})\big]^{1/2}$ from \eqref{hessian}, we get
		\[
		\lim\limits_{t \to \infty}\sup\limits_{\|x-y\| \leq tL_0}	\Big|\frac{1}{\ph_{\hat{\zeta}}(y)\ph^*_{\hat{\zeta}}(y)}[\det D^2\Phi\big(\frac{y-x}{t}\big)\big]^{-1/2}(\sqrt{2\pi t})^d  p^{\hat{\zeta}}(t,x ,y) - 1 \Big|= 0.
		\]

Finally, we establish the claim \eqref{philower}. If this is not the case, then there must be sequences $\{y_n\}_{n=1}^\infty \subset [0,1]^d$ and $\{ \zeta_n \}_{n=1}^\infty$, with $\|\zeta_n\|\leq L$, such that
\[
\lim_{n \to \infty} \ph_{\zeta_n}(y_n) \ph^*_{\zeta_n}(y_n) = 0.
\]
Since $\zeta$ and $y$ are confined to a compact set, we can extract a subsequence of the pairs $\{(y_n,\zeta_n)\}_{n=1}^\infty$ that converges to some $(\bar y,\bar \zeta)$. Joint continuity of $(y,\zeta) \mapsto \ph_{\zeta}(y) \ph^*_{\zeta}(y)$ implies that $\ph_{\bar \zeta}(\bar y) \ph^*_{\bar \zeta}(\bar y) = 0$, although the normalization \eqref{noreig} holds for $\ph_{\bar \zeta}$ and $\ph^*_{\bar \zeta}$. This is a contradiction, since the periodic principal eigenfunctions of elliptic operators $\mathcal{L}$ and $\mathcal{L}^*$ have a strict sign. We conclude that \eqref{philower} holds.

\end{proof}

To complete the proof of Proposition \ref{prop:kasymp}, we now prove Lemma \ref{lem222}. This follows an argument of Hennion and Herv\'e  \cite{HH} where a very similar lemma was proved (see  Lemma VI.4 of \cite{HH}) in the discrete time one dimensional setting; we will explain the technical differences in Remark \ref{rem:diff} below.

Let us define 
\[
(S^1)^d = (\real/(2 \pi \mathbb{Z}))^d = \{  (\theta_1, \theta_2, \cdots, \theta_d)\;|\;  \theta_i \in \real/(2\pi \mathbb{Z}), \;\; i = 1,\dots,d\}.
\]
For $g: \bbZ^d \to \real$, for $\theta \in(S^1)^d$ , $z\in \bbZ^d$, we use the following definitions of Fourier Transform and Inverse Fourier Transform:
\[
\hat g(\theta) = \frac{1}{(\sqrt{2\pi})^d}\sum_{z\in \bbZ^d} g(z)e^{i\theta z}, \quad \quad g(z) = \frac{1}{(\sqrt{2\pi})^d} \int_{(S^1)^d}\hat g(\theta) e^{-i\theta z} d\theta.
\]
Letting $\hat{g}(-\theta) = \tilde g(\theta)$, we have
\begin{align}\label{fif}
(\sqrt{2\pi t})^d\EXP_x[f(X_t) g([X_t])] & = t^{d/2}\EXP_x[f(X_t) \int_{(S^1)^d}\tilde g(\theta) e^{i \theta [X_t]} \,d\theta] \no \\
& = t^{d/2}\int_{(S^1)^d}\tilde g(\theta) \EXP_x[f(X_t)  e^{i \theta [X_t]} ] \,d \theta.
\end{align}
For $\theta \in \real^d$, $t \geq 0$, let us define the {\em Fourier Kernels} $Q_\zeta(\theta,t)$, acting on $\mathcal{B}$, by 
\begin{equation}\label{fq}
Q_\zeta(\theta, t)f(x) = \EXP_x^\zeta (f(X_t)e^{i \theta ([X_t]- [x]- \ell(\zeta)t)}).
\end{equation}

Now recalling the definition \eqref{mchidef}, observe that
\begin{align}
m_t^{\chi}(g) & = \det(\Xi_\zeta)^{1/2}(\sqrt{2\pi t})^d \int_{\mathbb{R}^d} p^\zeta(t,x,y)f(y) g([y] - z)) \,dy \nonumber \\
& =\det(\Xi_\zeta)^{1/2}(\sqrt{2\pi t})^d \EXP_x^\zeta(f(X_t)g([X_t] - z)) . \label{mchig1}
\end{align}
Using the Fourier inversion formula and Fubini's theorem, \eqref{mchig1} can be written as
\[
m_t^{\chi}(g) =\det(\Xi_\zeta)^{1/2}t^{d/2} \int_{{(S^1)^d}} \tilde{g}(\theta)e^{-i\theta(z -\ell(\zeta)t - [x])} (Q_\zeta(\theta,t)f)(x)  d\theta,
\]
where $\tilde{g}(\theta):= \hat{g}(-\theta)$.

The Fourier kernels $\{Q_\zeta(\theta,t)\}_{t \geq 0}$ are a family of compact operators on $\mathcal{B}$ and $e^{i\theta \ell(\zeta)t}Q_\zeta(\theta, t)$ is $2\pi \bbZ^d$ periodic in the parameter~$\theta$.  One can show that for a fixed $\theta \in (S^1)^d$, the family $\{Q_\zeta(\theta,\cdot)\}_{t \geq 0}$ forms a semigroup. That is, for each $x\in  \real^d$, $t,s \geq 0$,
\begin{equation}\label{semigroup}
Q_\zeta(\theta, t)\circ Q_\zeta(\theta,s)f(x) = Q_\zeta(\theta, t+ s)f(x).
\end{equation}
Observe that, for $\theta = 0$, $Q_\zeta(0,t)$ is the Markov operator corresponding to the process $X_t$, which is generated by $\cK_\zeta$. Therefore, since zero is the principal simple eigenvalue of the operator $\cK_\zeta $, $1$ is the principal simple eigenvalue of the operator $Q_\zeta(0,t)$. By a perturbation theorem (see, for example, Theorem III.8 in \cite{HH}), there exists a small $\theta_0>0$ such that, for each $\theta \in (S^1)^d$ with $\|\theta\|\leq \theta_0$, the principal eigenvalues of the operators $Q_\zeta(\theta,1)$ are simple, for each $\|\zeta\|\leq L$. We denote these principle eigenvalues of the operators $Q_\zeta(\theta,1)$ by $\lambda(\zeta,\theta) \in \complex$, for $\|\theta\|\leq \theta_0$. Thus, from the semigroup property \eqref{semigroup} and the time homogeneity of the coefficients of the partial differential operator $\cK_\zeta$, we conclude that the principal eigenvalue of the operator $Q_\zeta(\theta,t)$ is $\lambda(\zeta,\theta)^t$ for each $t\geq0$.

The proof of Lemma \ref{lem222} is based on the following spectral decomposition of the operator $Q_\zeta(\theta,t)$.

\begin{lemma}\label{lem111}
For a fixed $L >0$, there exist $\theta_0>0$, $q > 0$, and $\eta >0$ such that, for each $t> 1$, $\theta \in (S^1)^d$ with $\|\theta\| < \theta_0$,  $f \in \mathcal{B}$, and $|\zeta| \leq L$ we have 
\begin{equation}\label{kii}
Q_\zeta(\theta, t)f(x) =  \lambda(\zeta, \theta)^t\big[\langle \ph_\zeta\ph_\zeta^*, f \rangle  +  (M_\zeta(\theta,t)f)(x)\big] +  (N_\zeta(\theta,t)f)(x), 
\end{equation}
where the following bounds for the operator $M_\zeta(\theta,t)$ and $N_\zeta(\theta,t)$ hold:
\begin{equation}\label{Ln}
\|M_\zeta(\theta,t) f\|_{L^\infty} \leq q\|f\|_{\infty}\|\theta\|,~~~~~~~
\|N_\zeta(\theta, t)f\|_{L^\infty}  \leq qe^{-\eta t}\|f\|_{\infty},
\end{equation}
Moreover, there exists a constant $C_1$ such that for each $\theta \in (S^1)^d$ with $\|\frac{\theta}{\sqrt{ t}}\| \leq \theta_0$ we have
\begin{equation}\label{eqk1}
\| \lambda(\zeta, \frac{\theta}{\sqrt{t}})^t - e^{-\frac{\theta^T\Xi_\zeta\cdot \theta}{2}}\| \leq \frac{C_1}{\sqrt t}\|\theta\|^3e^{\frac{-\theta^T\Xi_\zeta \cdot \theta}{4}},
\end{equation}
uniformly over $|\zeta| \leq L$. 
\end{lemma}

\begin{proof}[Proof of Lemma \eqref{lem111}]
In the discrete time one dimensional setting, Lemma \eqref{lem111} is proved in Hennion and Herv\'e \cite{HH} (see Proposition VI.2, therein), but the arguments there also go through in the continuous time $d$-dimensional setting.   The assumptions of that Proposition, denoted by $H''$[2] in \cite{HH} (assumptions on the Banach space being sufficiently big, $Q_\zeta(0,1)$ having $1$ as its simple eigenvalue corresponding to the eigenfunction $f \equiv 1$, and the operators $Q_\zeta(\theta,1)$ being sufficiently regular in the variable $\theta$ in a small neighborhood around $\theta = 0$) are all satisfied in our setting, uniformly in $\|\zeta\|\leq L$. The proof of \eqref{eqk1} (or, rather, its analog in \cite{HH}) relies on the fact that $\nabla_\theta\lambda(\zeta,\theta)\big|_{\theta = 0} = 0$ and $D_\theta^2\lambda(\zeta,\theta)\big|_{\theta = 0} = - \Xi_\zeta$, which follows from arguments similar to those used in proving \eqref{hessian}. 
\end{proof}

To apply Lemma \ref{lem111} in the proof of Lemma \ref{lem222}, we will need the following fact about the eigenvalues $\lambda(\zeta,\theta)$. For a bounded linear operator $Q$ on Banach space $\mathcal{B}$, let $r(Q)$ denote its spectral radius. 

\begin{lemma}\label{lem:eigs} 
For each $\theta_0 \in (0, 2\pi)$, $L>0$, 
\[
\alpha(\theta_0,L) := \sup \{ r(Q_\zeta(\theta,1)) \;|\; \|\zeta\| \leq L,\;\; \theta \in (S^1)^d, \;\; \|\theta\| \geq \theta_0   \} < 1.
\]
\end{lemma}

\begin{proof}[Proof of Lemma \ref{lem:eigs}]
From the definition of the operators $Q_\zeta(\theta,1)$ and using the fact that $e^{i\theta \ell(\zeta)t}Q_\zeta(\theta, t)$ is $2\pi \bbZ^d$ periodic in the parameter~$\theta$, we have, for a fixed $\zeta \in \real^d$, the function $r(Q_\zeta(\theta,1))$ is continuous in the variable $\theta \in (S^1)^d$. Let us fix $\zeta\in \real^d$ with  $\|\zeta\|\leq L$. It is clear that $r(Q_\zeta(\theta,1)) \leq 1$ for each $\theta\in (S^1)^d$. Indeed, if $f\in \mathcal{B}$ with $\|f\| = 1$, \begin{align*}
    \|Q_\zeta(\theta,t)f\| &= \|\EXP_x(f(X_t)e^{i\theta([X_t] - [x] - \ell(\zeta))})\|\\
    &\leq \|\EXP_x(|f(X_t)|) \| \\ 
    & = \|~Q_\zeta(0,t)|f|~\| \leq 1.
\end{align*}  
That is, if $\eta(\zeta, \theta)\in \complex$ is any eigenvalue of the operator $Q_\zeta(\theta,1)$, $|\eta(\zeta, \theta)| \leq 1$ for all $\theta \in (S^1)^d$. Now for $\theta \in (0,2 \pi)$, suppose that, there exists an eigenfunction $f \in \mathcal{B}$ of the operator $Q_\zeta(\theta, t)$ with $\|f\| = 1$ corresponding to the eigenvalue $\eta(\zeta, \theta) \in \complex$ such that $|\eta(\zeta, \theta)| = 1$. That is, for each $x \in [0,1)^d$,
\begin{equation}
	|\EXP_{x}( f(X_t)e^{i\theta([X_t] - [x] - \ell(\zeta)t ) })|= |f(x)|.
\end{equation}
We know that $1$ is the simple principal eigenvalue of the operator $Q_\zeta(0,t)$. Thus, there exists an eigenfunction $g \in \mathcal{B}$ of $Q_\zeta(0,t)$ such that $g$ is strictly positive and 
\begin{equation}
	\EXP_{x} (g(X_t) )= g(x), \quad x \in [0,1)^d.
\end{equation}
Since $g > 0$, we can multiply $g$ by a constant so that $|f(x)| \leq g(x)$ holds for all $x \in [0,1)^d$ with equality $|f(x_0)| = g(x_0)$ holding at some point $x_0 \in [0,1)^d$. Now, 
\[
\EXP_{x_0}(\big|f(X_t)e^{i\theta([X_t] - [x] - \ell(\zeta)t)}\big|)  \geq |\EXP_{x_0}( f(X_t)e^{i\theta([X_t] - [x] - \ell(\zeta)t)})|\]\[ = |f(x_0)| = g(x_0) = \EXP_{x_0}g(X_t). 
\]
This implies that,
\[
\EXP_{x_0}(|f(X_t)| - g(X_t)) = Q_\zeta(0,t)(|f| - g)(x_0)\geq 0.
\]
Since $|f| \leq g$ and $Q_\zeta(0,t)$ is a positive operator, we conclude that \[\EXP_{x_0}(|f(X_t)| - g(X_t)) =0.\] That is,
\[
\int_{\real^d}(|f(y)| - g(y))p^\zeta(t,x_0,y)dy = 0.
\]
Since $X_t$ is a non-degenerate diffusion, for a fixed $x_0 \in [0,1)^d$, $p^\zeta(t,x_0,y) >0$ for all $y \in \real^d$, $t \geq 0$. Thus, there exists a continuous $\bbZ^d$ periodic function $h$ such that $f(y) = e^{ih(y)}g(y)$ for all $y \in \real^d$. Therefore,
\[
\EXP_{x}( e^{ih(X_t)}g(X_t)e^{i\theta([X_t] - [x] - \ell(\zeta)t)}) = e^{ih(x)}g(x) = e^{ih(x)}\EXP_{x}( g(X_t)).
\]
Thus, 
\[
\EXP_{x}\Big( g(X_t)\Big[e^{i\big(\theta([X_t] - [x] - \ell(\zeta)t) + h(X_t) - h(x)\big)} - 1\Big]\Big) = 0,
\]
which implies that $\theta([y] - [x] - \ell(\zeta)t) + h(y) - h(x) \in 2\pi \bbZ$, for all $x,y \in \real^d$, $t \geq 0$. This is a contradiction since, taking $y = x+m$ with $m \in \bbZ^d$, we get $\theta (m -\ell(\zeta) t )\in 2\pi \bbZ$ for all $m \in \bbZ^d$, which is impossible. Thus we have shown that for each $\zeta \in \real^d$ with  $\|\zeta\|\leq L$, and each $\theta \in (S^1)^d$, with $ \|\theta\| \geq \theta_0 $, $|\eta(\zeta,
\theta)| < 1$.  Therefore, choosing $\alpha(\theta_0, L) = \sup\{r(Q_\zeta(\theta,1) \;|\; \|\zeta\| \leq L,\;\; \theta \in (S^1)^d, \;\; \|\theta\| \geq \theta_0 \}$, we get the required result.
\end{proof}

\begin{proof}[Proof of Lemma \ref{lem222}]

From Lemma \ref{lem111}, we know that there exists a $\theta_0 > 0$ such that, for all $\|\theta\|\leq \theta_0$ the decomposition \eqref{kii} holds. Therefore, we can write
\[
m_t^{\chi}(g) = J_t^1(\chi) + J_t^2(\chi) + J_t^3(\chi) ,
\] 
where
\[
J_t^1(\chi) :=\]\[= \det(\Xi_\zeta)^{1/2}t^{d/2}\int_{(S^1)^d\cap (\|\theta\|<\theta_0)}\tilde{g}(\theta)e^{-i\theta(z - \ell(\zeta)t - [x])} \lambda(\zeta, \theta)^t\big[\langle \ph_\zeta\ph_\zeta^*, f \rangle +  M_\zeta(\theta, t)f(x)\big]  d\theta
\]
and $J_t^2(\chi)$ and $J_t^3(\chi)$, are defined by
\[
J_t^2(\chi) :=\det(\Xi_\zeta)^{1/2}t^{d/2}\int_{(S^1)^d \cap  (\|\theta\|<\theta_0)}\tilde{g}(\theta)e^{-i\theta(z - \ell(\zeta)t- [x])} N_\zeta(\theta, t)f(x) d\theta,
\]
and
\[
J_t^3(\chi): =\det(\Xi_\zeta)^{1/2}t^{d/2}\int_{(S^1)^d\cap (\|\theta\|\geq \theta_0)}\tilde{g}(\theta)e^{-i\theta(z - \ell(\zeta)t -[x])} Q_\zeta(\theta, t)f(x)d\theta.
\]

We claim that as $t \to \infty$,
\[
\|J_t^1(\chi) - \bar m_t^{\chi}(g) \|  \to 0,
\]
and
\[
\| J_t^3(\chi)\| \to 0, \quad \quad \text{and} \quad \quad \| J_t^3(\chi)\| \to 0,
\]
uniformly over $\chi \in (\bbZ^d, \real^d, \mathcal B_{+,r}, B_0(L))$. The change of variable  $\theta = \frac{s}{\sqrt{ t}}$ gives
\[
J_t^1(\chi) =  \int_{\real^d} k_t(s)e^{-i(\frac{s(z - \ell(\zeta)t  -[x])}{\sqrt{t}})} \big[\langle \ph \ph_\zeta^*, f \rangle  +  M_\zeta(\frac{s}{\sqrt{ t}}, t)f(x)\big]  ds,
\]
where
\[
k_t(s) =\det(\Xi_\zeta)^{1/2} {{\bf{1}}}_{{{(S^1)^d}} \cap (\|\theta\|<\theta_0)}(\frac{s}{\sqrt{ t}})\tilde{g}(\frac{s}{\sqrt{t}})\lambda(\zeta, \frac{s}{\sqrt{ t}})^t.
\] 
On the other hand, we have
\[
\bar m_t^{\chi}(g) = \int_{\real^d} e^{-i\frac{s(z - \ell(\zeta)t-[x])}{\sqrt{t}}}k(s)\langle \ph_\zeta\ph_\zeta^*, f \rangle ds,
\]
where,
\[
k(s) :=\det(\Xi_\zeta)^{1/2}~\tilde{g}(0)e^{-\frac{s^T\Xi_\zeta s}{2}}.
\]
For each $s\in (S^1)^d$ such that $\|\frac{s}{\sqrt{ t}}\| <\theta_0$, from Lemma \ref{lem111}, we have that 
\[
\|M_\zeta(\frac{s}{\sqrt{ t}}, t) f(x)\| \leq q \|f\| \frac{\|s\|}{\sqrt{ t}}.
\]
Hence,
\[
\|J_t^1(\chi) - \bar m_t^{\chi}(g) \| \leq |\langle \ph_\zeta\ph_\zeta^*, f \rangle| \int_{\real^d}  |k_t(s) - k(s)|ds + q\|f\|\int_{\real^d}  |k_t(s)|\frac{\|s\|}{\sqrt{t}}ds.
\]
We observe from (\ref{eqk1}) that the sequence $\{k_t\}_{t\geq 1}$ converges point-wise to $k$. Since the function $g$ has bounded support in $\bbZ^d$,  $\|\tilde g\|_\infty <\infty$. Thus, setting $c_g:= \|\tilde{g}\|_{\infty}$, we have
\[\|k_t(s)\| \leq \det(\Xi_\zeta) c_g e^{-\frac{s^T\Xi_\zeta s}{4}}.
\]	
By defining 
\[
\epsilon_t^1 := \|\ph_\zeta\ph_\zeta^*\| \int_{\real^d}  |k_t(s) - k(s)|ds,~~~~~~  \epsilon_t^{2} :=  q \int_{\real^d}  |k_t(s)|\frac{\|s\|}{\sqrt{ t}}ds,
\]
we get, 
\[
\|J_t^1(\chi) - \bar m_t^{\chi}(g) \| \leq (\epsilon_t^1+ \epsilon_t^{2} )\|f\|.
\]
Using the Lebesgue dominated convergence theorem, $\lim\limits_{t \to \infty} \epsilon_t^1 = \lim\limits_{t \to \infty} \epsilon_t^{2} = 0$, uniform over $\chi \in (\bbZ^d, \real^d, \mathcal B_{+,r}, B_0(L))$. Now it remains to consider the terms $J_t^2(\chi)$ and $J_t^3(\chi)$.
For $ \|\theta\| \leq \theta_0$, we have from Lemma \ref{lem111} that $\|N_\zeta(\theta, t)\| \leq q e^{-\eta t}$, and therefore 
\[
J_t^2(\chi) \leq \det(\Xi_\zeta)^{1/2} t^{d/2}qe^{-\eta t}\|f\|\int_{(S^1)^d \cap  (\|\theta\|<\theta_0)} |\tilde{g}(\theta)| d\theta =: \epsilon_t^{3}\|f\|,
\]
where 
\[
\epsilon_t^{3}:=\det(\Xi_\zeta)^{1/2}t^{d/2}qe^{-\eta t}\int_{(S^1)^d \cap  (\|\theta\|<\theta_0)} |\tilde{g}(\theta)| d\theta .
\]  
It is clear that  $\lim\limits_{t \to \infty} \epsilon_t^{3}= 0$, uniformly over $\chi \in (\bbZ^d, \real^d, \mathcal B_{+,r}, B_0(L))$.  Let $\beta_t = \sup\{\|Q_\zeta(\theta, t)\| : \theta \in (\|\theta\| \geq \theta_0)\cap (S^1)^d,  \|\zeta\|\leq L\}$.  
 From Lemma \ref{lem:eigs}, by choosing 
\[
\alpha(\theta_0,L) = \sup \left\{r(Q_\zeta(\theta, 1)) \Big| \theta \in (S^1)^d, \|\theta\|\geq \theta_0, \|\zeta\|\leq L \right\} <1,
\]
we now have $\beta_t \leq \alpha(\theta_0,L)^t \to 0$ exponentially fast, as $t \to \infty$. Now,
\[
\|J_t^3(\chi)\| \leq\det(\Xi_\zeta)^{1/2} t^{d/2}\|f\| \beta_t \int_{{{(S^1)^d}}}\tilde{g}(\theta) d\theta =:  \epsilon_t^{4}\|f\|,
\]
where 
\[
\epsilon_t^{4} := \det(\Xi_\zeta)^{1/2} t^{d/2} \beta_t \int_{{{(S^1)^d}}}\tilde{g}(\theta) d\theta.
\]
It is clear to see that $\lim\limits_{t \to \infty} \epsilon_t^{4} = 0$, uniformly over $\chi \in (\bbZ^d, \real^d, \mathcal B_{+,r}, B_0(L))$.  Combining these estimates, we conclude that 
\[
\lim\limits_{t \to \infty}\sup\limits_{\chi} \|m_t^{\chi}(g) - m_t^{'\chi}(g)\| = 0.
\]
This concludes the proof of Lemma \ref{lem222}.

\end{proof}

\begin{remark} \label{rem:diff}
As we have mentioned, the above proof of Lemma \ref{lem222} follows very closely the proof of Lemma VI.4 of \cite{HH}.  The difference is that in Lemma VI.4 of \cite{HH} the set $\{\theta \in \real^d\big| r(Q_\zeta(\theta,1))\geq 1 \}$ was required to be $\{0\}$. This condition does not hold in our setting since the operators $Q(\theta,1)e^{i\theta\ell(\zeta)}$ are $2\pi \bbZ^d$ periodic in $\theta\in \real^d$. Instead, we have shown in Lemma \ref{lem:eigs} that $\{\theta \in (S^1)^d\big| r(Q_\zeta(\theta,1))\geq 1 \} = \{0\}$. Another difference is that, in our setting, the operators $Q$ also vary with respect to the additional parameter $\zeta\in \real^d$. 
\end{remark}

		\section{Proof of \thmref{thm:nof}}

		The main idea of the proof is to look at the higher order correlation functions and the corresponding PDEs they solve and then use the asymptotics of the density function obtained in Theorem \ref{thmmain1} and techniques developed in \cite{KorMol} to obtain logarithmic asymptotics of the moments $\EXP(n^{t{\bv}}(t,x)^k)$.

			Recall that $\bar{\bv} =  \ell(0) = \nabla\mu(0)$ is the effective drift of the branching process defined at \eqref{effectivedrift} (also see Lemma \ref{lambda}), and $\Phi(\bar{\bv}) = -\mu(0)$. Without loss of generality, we may assume that $\bar{\bv} = 0$, which simplifies our notation.  Let $B_\delta(y)$ denote a ball of radius $\delta>0$ centered at $y\in \real^d$. For $t > 0$ and $x, y_1, y_2, ... \in \real^d$ with all $y_i$ distinct, define the particle density $\rho_1(t,x,y)$ and the higher order correlation functions $\rho_n(t,x,y_1,...., y_n)$ as the limits of probabilities of finding $n$ distinct particles in $B_\delta(y_1) ,... B_\delta(y_n) $, respectively, divided by the $n$-th power of the volume of  $B_\delta(0) \subset\real^d$. For a fixed $y_1$, the density satisfies 
		\begin{equation}\label{rho1}
			\partial_t\rho_1(t,x,y_1) = \cL_x \rho_1(t,x,y_1), ~~~~~ \rho_1(0,x,y_1) = \delta_{y_1}(x),
		\end{equation}
where $\cL_x$ is the linear operator defined at \eqref{gener0}, acting on the variable $x$. The equations on $\rho_n$ , $n > 1$, are as follows
		\begin{equation}\label{rhon}
			\partial_t\rho_n(t,x,y_1,y_2,...,y_n) = \cL_x \rho_n(t,x,y_1,y_2,...,y_n) + \alpha(x) H_n(t,x,y_1,y_2,...,y_n),
		\end{equation}

		\[ \rho_n(0,x,y_1,y_2,...,y_n) \equiv 0,
		\]
		where 
		\[
		H_n(t,x,y_1,y_2,...,y_n) = \sum\limits_{U \subset Y, U \neq \emptyset} \rho_{|U|} (t, x, U)\rho_{n-|U|}(t, x, Y\setminus U),
		\]
		where $Y = (y_1 , ..., y_n )$, $U$ is a proper non-empty subsequence of $Y$ , and $|U|$ is the number of elements in this subsequence. See Section 2 of \cite{KorMol}, for a derivation of these equations.

Define $m_k^y(t,x) = \int_{\cube_{y}^d}....\int_{\cube_{y}^d} \rho_k(t,x,y_1,y_2,...,y_k) dy_1...dy_k$.  By integrating \eqref{rhon}, it follows that 
		\begin{equation}\label{k1}
			\partial_t m_1^y(t,x) = \cL_xm_1^y(t,x), ~~~~m_1^y(0,x) = \chi_{\cube_{y}^d}(x),
		\end{equation}
		while for $k \geq 2$,
		\begin{equation}\label{k2}
			\partial_t m_k^y(t,x) = \cL_xm_k^y(t,x) + \alpha(x) \sum_{i = 1}^{k-1}\beta_i^km_i^y(t,x)m_{k-i}^y(t,x), ~~~~m_k^y(0,x) \equiv 0,
		\end{equation}
		where $\beta_i^k = k!/(i!(k-i)!)$.	The functions $m_i^{t{\bv}}$ are related to the moments $\EXP(n^{t{\bv}}(t,x)^k)$ according to 
\begin{align}
		\EXP(n^{t{\bv}}(t,x)^k) & = \sum_{i = 1}^{k}S(k,i)\int_{\cube_{t\bv}^d}..\int_{\cube_{t\bv}^d} \rho_i(t,x,y_1,y_2,..,y_i) dy_1...dy_i \nonumber\\
& = \sum_{i = 1}^{k}S(k,i)m_i^{t{\bv}}(t,x),\label{kthmoment}
\end{align}
 where $S(k,i)$ is the Stirling number of the second kind (the number of ways to partition $k$ elements into $i$ nonempty subsets).	As explained in Section 9 of \cite{KorMol}, this follows by partitioning $\mathcal{Q}^d_{t{\bv}}$ into small subdomains, and taking a limit as their diameters shrink uniformly to zero.

		\subsection{Proof of part (a).}  We first proof part (a) of Theorem \ref{thm:nof}.  We will use induction to show the following:	
		\begin{itemize}
			\item[(i)] For each $k \geq 1$, there exists a constant $a_k>0$ such that
			\begin{equation}\label{aronsonk}
				m^y_k(t,x) \leq a_k \exp\left(a_k t  - \frac{\|y-x\|^2}{a_k(t+1)}\right)
			\end{equation}	 
			for all $(t,x,y) \in \real^+ \times \real^d\times \real^d$.
			\item[(ii)]For each $k \geq 1$, for each $L>0$, the following two limits exist uniformly for  $\bv \in \real^d$, with $\|\bv\| \leq L$ and for $x \in [0,1)^d$, and satisfy
			\begin{equation}\label{limeq}
				\gamma_k(\bv) =	 \lim\limits_{t \to \infty} \frac{\ln m_k^{t\bv}(t,x)}{t} = \lim\limits_{t \to \infty} \frac{\ln\EXP(n^{t{\bv}}(t,x)^k)}{t}.
			\end{equation}
			Moreover, $\gamma_k:\real^d \to \real$ is continuous for all $k \in \naturals$. 
			\item[(iii)] For each $L>0$, there exists $M_0 = M_0(L,k)$ such that, for all $M\geq M_0$,
			\begin{equation}\label{oopp}
				\gamma_k(\bv) = \sup\limits_{\|w -  \bv\|\leq M,\\ u \in (0,1)} \Big[ u \gamma_{k-1}\Big(\frac{\bv - w}{u}\Big) +  u \gamma_{1}\Big(\frac{\bv - w}{u}
				\Big) + (1-u)\gamma_1\Big(\frac{ w}{1-u}\Big)\Big],
			\end{equation}
			when $\|\bv\|\leq L$, $k\geq 2$. In addition, $\gamma_k(\bv) \geq \gamma_{k-1}(\bv)$ for $k\geq 2$.
		\end{itemize}

Starting with $k = 1$, we estimate 
\begin{equation}
m_1^y(t,x) = \int_{\cube_{y}^d}\rho_1(t,x,z)\,dz. \label{m1integral}
\end{equation}
By Lemma \ref{aronlem}, we know that there is $c > 0$, such that 
\begin{equation}
\rho_1(t,x,y ) \leq 	ct^{-d/2}\exp\left((-\Phi(\bar {\bv}))t - \frac{\| (y - x) -  \bar{\bv}t \|^2}{ct}\right), \quad x,y \in \real^d, \;\;t \geq 0. \label{rho}
\end{equation}
In view of \eqref{m1integral} and the inequality $-\|y - x - \bar \bv\|^2 \leq - \frac{1}{2}\|y - x\|^2 + \frac{1}{2} \|\bar \bv\|^2 t^2$, this implies
	\begin{equation}\label{aronsonm1}
			m_1^y(t,x) \leq a_1 \exp\left(a_1t - \frac{\|y-x\|^2}{a_1(t+1)}\right), \quad x,y \in \real^d, \;\;t \geq 0
		\end{equation}
holds for some constant $a_1 > 0$. 	This proves (i) for $k = 1$.

Now suppose that (i) holds up-to $k-1$. From $(\ref{k2})$ and the Duhamel's Formula, we see that 
		\begin{equation}\label{exactmk}
			m_k^y(t,x) =\int_0^t \int_{\real^d} \alpha(z) \sum_{i = 1}^{k-1}\beta_i^km_i^y(s,z)m_{k-i}^y(s,z) \rho_1(t-s,x,z)dz ds.
		\end{equation}
		Note that, since $\eqref{aronsonk}$ holds up-to $k-1$, it also holds for $\sum_{i = 1}^{k-1}\beta_i^km_i^y(s,z)m_{k-i}^y(s,z)$ (with a different constant $\tilde{a}_{k-1}$).
		Thus there exists a constant $a_k>0$ such that $m^y_k(t,x) \leq a_k \exp\left(a_k t  - \frac{\|y-x\|^2}{a_k(t+1)}\right)$, since the convolution of two functions satisfying the estimate \eqref{aronsonk}, with two different constants also satisfies \eqref{aronsonk}. That is, (i) holds for $k$, as well.

		We next show that (ii) holds for $k = 1$ and $k = 2$, and (iii) holds for $k = 2$. Here is where we will need the sharp estimate for $\rho_1$, provided by Theorem \ref{thmmain1}: for any fixed $L>0$ and for all $x,y \in \real^d$ with $\|x-y\|\leq Lt$, we have  
\begin{equation}\label{formm1} 
\rho_1(t,x,z) = (\sqrt{2\pi t})^{-d}\ph_{0}(x)	\text{det}[ D^2 \Phi(\frac{z-x}{t}) ]^{1/2}e^{-t\Phi(\frac{z-x}{t})} \ph^*_{0}(z)\left[1 +o_L(1)\right],
		\end{equation}
where $\Phi$, $\ph_{0}$ and $\ph_{0}^*$ are defined before Theorem \ref{thmmain1}.  From  \eqref{formm1} and \eqref{m1integral}, we obtain 
		\begin{equation}\label{m1asy}
			\gamma_1(\bv) = \lim\limits_{t \to \infty}\frac{\ln m_1^{t\bv}(t,x)}{t} = -\Phi(\bv),
		\end{equation}
		and $\gamma_1$ is continuous since $\Phi$ is continuous.
		In addition, from \eqref{kthmoment}, for each $t >0$, $\EXP(n^{t\bv}(t,x)) = m_1^{t\bv}(t,x)$. Thus (ii) holds for $k = 1$. \\
		\\
		Next we show that, for $k = 2$, the first limit on the right hand side of \eqref{limeq} exists and satisfies formula \eqref{oopp}. In the arguments below, we treat $x$ and $\bv$ as fixed, but all the estimates are easily seen to be uniform in $\|\bv\|\leq L$ and $x \in [0,1)^d$.
		Let us recall that 
		\[
		m_2^y(t,x) =  \int_0^t \int_{\real^d} 2\alpha(z)(m_1^y(s,z))^2 \rho_1(t-s,x,z)dz ds.
		\]
We will apply Laplace's method to estimate the integral.  For $0< \eps<\kappa<1$ and $M > 0$, consider the following partition of the domain $[0,t] \times \real^d$:
\begin{equation}
[0,t] \times \real^d = R_1 \cup R_2 \cup R_3 \cup R_4 \cup R_5 \cup R_6 \label{domainPartition1}
\end{equation}
with
\begin{align}
R_1 & = [0,\epsilon t) \times \{\|z - t\bv\| > \eps^{1/4}t \}  \no \\
R_2 & = [0,\epsilon t) \times \{\|z - t\bv\|\leq \eps^{1/4}t \}  \no \\
R_3 & = [\epsilon t, \kappa t] \times \{\|z - t\bv\| > M t \}  \label{domainPartition2} \\
R_4 & = [\epsilon t, \kappa t] \times \{\|z - t\bv\| \leq M t \}  \no \\
R_5 & = (\kappa t, t] \times \{\|z - t\bv\| > (1 - \kappa)^{1/4} t \} \no \\
R_6 & = (\kappa t, t] \times \{\|z - t\bv\| \leq (1 - \kappa)^{1/4} t \}. \no
\end{align}
Then we write
		\[
		m_2^y(t,x) =  \sum_{j=1}^6 I_j, \quad \quad I_j = \int \!\! \int_{R_j} 2\alpha(z)(m_1^y(s,z))^2 \rho_1(t-s,x,z)dz ds.
\]
Using \eqref{aronsonm1} in the region $R_1$, where $s < \eps t$ and $\|\bv t - z\| > \eps^{1/4}t$, we see that
		\[
		( m_1^{\bv t}(s,z))^2 \leq a_1^2\exp(2(a_1\eps- \frac{\sqrt{\eps}}{a_1(\eps +\frac{1}{t})})t), \quad (s,z) \in R_1,
		\] 
		which can be made exponentially small (as $t \to \infty$), with an arbitrarily large negative exponent, by choosing $\eps$ small enough. Therefore, using the estimate on $\rho_1$ from \eqref{rho}, we infer that for each $r>0$, for all sufficiently small $\eps>0$, 
		\[
		\limsup\limits_{t \to \infty} \frac{\ln I_1(t,x,t\bv)}{t} \leq  -r.
		\]
		Similarly, considering the integral $I_5$ over the region $R_5$, we may exchange the roles of $(m_1^{t\bv}(s,z))^2$ and $\rho_1(t-s,z,x)$, to obtain for each $r>0$, for all $\kappa \in (0,1)$ sufficiently close to $1$,
		\[
		\limsup\limits_{t \to \infty} \frac{\ln I_5(t,x,t\bv)}{t} \leq  -r.
		\]
		In the region $R_2$, where $s < \eps t, \|\bv t - z\| < \eps^{1/4}t$, using \eqref{aronsonm1}, we conclude that there exists a $C_1>0$ such that 
		\[
		( m_1^{\bv t}(s,z))^2  \leq C_1 e^{C_1\eps t}, \quad \forall \;(s,z) \in R_2.
		\] 
		Moreover, by \thmref{thmmain1}, there exists $C_2>0$ such that 
\[
\rho_1(t-s, x,z)\leq C_2(t-s)^{-d/2} e^{-(t-s)\Phi(\frac{z-x}{t-s})}, \quad \forall \;(s,z) \in R_2.
\]
 By choosing $\eps>0$ small enough, and choosing sufficiently large $t$, the value of $-\Phi(\frac{z-x}{t-s})$ in this region $R_2$ can be made arbitrarily close to $\gamma_1(\bv)$. Thus, for each $\delta>0$, for all sufficiently small $\eps>0$,
		\[
		\limsup\limits_{t \to \infty} \frac{\ln I_2(t,x,t\bv)}{t} \leq \gamma_1(\bv) + \delta.
		\]
		Similarly, considering the integral $I_6$ over the region $R_6$, we may exchange the roles of $(m_1^{t\bv}(s,z))^2$ and $\rho_1(t-s,z,x)$ to obtain for each $\delta>0$, for $\kappa \in (0,1)$ sufficiently close to $1$,
		\[
		\limsup\limits_{t \to \infty} \frac{\ln I_6(t,x,t\bv)}{t} \leq  2\gamma_1(\bv) + \delta.
		\]
		Now let us assume $1>\kappa>\eps>0$ are fixed. Consider the integral over the outer region $R_3$. From \eqref{aronsonm1},  it follows that, given $r>0$, we can choose $M$ large enough such that
		\[
		\limsup\limits_{t \to \infty} \frac{\ln I_3(t,x, t\bv)}{t} < -r.
		\]
		Let us now examine the asymptotics of $I_4 = I_4(t,x, t\bv)$, the integral over $R_4$:
\[
I_4(t,x, t\bv) = \int_{\epsilon t}^{\kappa t} \int_{\|z - t \bv\| \leq M t} 2\alpha(z)(m_1^{\bv t}(s,z))^2 \rho_1(t-s,x,z)dz ds
\] 
Changing variables $s/t = u$ and $z/t = w$, this is equivalent to
\[
I_4(t,x, t\bv) = t^{ d + 1}\int_{\epsilon }^{\kappa } \int_{\|w - \bv\| \leq M } 2\alpha(wt)(m_1^{\bv t}(ut,wt))^2 \rho_1(t (1 - u),x,tw )dw du
\] 
The asymptotic behavior of $m_1$ and $\rho_1$ is available in $\eqref{m1asy}$ and Theorem $\ref{thmmain1}$. Observe that $\alpha(x)$ is periodic, non-negative and not identically $0$. Therefore, following Laplace's method, we have as $t \to \infty$,
\begin{align}
\frac{1}{t} \ln I_4(t,x, t\bv) & \sim \frac{1}{t} \ln \left( \int_{\epsilon }^{\kappa } \int_{\|w - \bv\| \leq M } \left(e^{- t u \Phi(\frac{\bv - w}{u}) } \right)^2 e^{- t (1 - u) \Phi(\frac{w - (x/t)}{1 -u})}  dw du \right) \\
& \sim   \sup\limits_{\|w-\bv\|\leq M, u \in (\eps,\kappa)} \Big[ 2u\gamma_{1}\Big(\frac{\bv - w}{u}
		\Big) + (1-u)\gamma_1\Big(\frac{ w}{1-u}\Big)\Big].
\end{align}
Combining this with the estimates on $I_1$, $I_2$, $I_3$, $I_5$, $I_6$, we obtain that the first limit in \eqref{limeq} exists for $k=2$, and is given by formula \eqref{oopp}:
		\begin{equation} \label{gamma2form}
		\gamma_2(\bv) := \lim\limits_{t \to \infty}\frac{\ln(m_2^{t\bv}(t,x))}{t} = \sup\limits_{\|w -  \bv\|\leq M,\\ u \in (0,1)} \Big[ 2u\gamma_{1}\Big(\frac{\bv - w}{u}
			\Big) + (1-u)\gamma_1\Big(\frac{ w}{1-u}\Big)\Big].
		\end{equation}
From the formula above, since $\gamma_1(v)$ is continuous, we conclude that $\gamma_2$ is also continuous. 

Next we show that $\gamma_2(\bv) \geq \gamma_1(\bv)$ for all $\bv \in \real^d$. This is complete the proof that (iii) holds for $k=2$. In view of \eqref{kthmoment}, this also implies that, for $k=2$, the second limit in \eqref{limeq} exists and is equal to $ \gamma_2(\bv)$. Recall that $m_1^y(t,x)$ and $m_2^y(t,x)$ solve the following PDEs:
		\begin{equation}\label{m1eq}
			\partial_t m_1^y(t,x) = \cL_xm_1^y(t,x), ~~~~m_1^y(0,x) = \chi_{\cube_{y}^d}(x),
		\end{equation}
		\begin{equation}\label{m2eq}
			\partial_t m_2^y(t,x) = \cL_xm_2^y(t,x) + \alpha(x)(m_1^y(t,x))^2, ~~~~m_2^y(0,x) \equiv 0.
		\end{equation}
		We will show that there exists a $C_L>0$ such that, for each $t\geq 1$  and $x,y\in \real^d$ with $\|x-y\|\leq Lt$, we have 
		\[
		m_2^y(t,x) \geq C_L m_1^y(t,x).
		\]
		Fix $R>0$ such that $[0,1)^d \in B_R(0)$. Observe that, since $m_1^y(0,x) = \chi_{\cube_{y}^d}(x)$, there exists a $\delta_1>0$ such that
		\[
		m_1^y(t,x) \geq \delta_1\chi_{B_R(y)}(x) ~~~~~\text{for all}~~~~ t \in [1/8,1/4].
		\]
		Also observe that there exists a $\delta_2>0$ such that, for all $x,y \in \real^d$ with $\|x-y\|\leq 2R$ and  $t \in [1/4,1/2]$,
		\[
		\rho_1(t,x,y) \geq \delta_2.
		\]
		Now, observe that $\alpha(x)$ is periodic, non-negative and not identically $0$. Thus, from \eqref{m2eq}, using Duhamel's Formula, for $x \in \cube_{y}^d$, 
		\begin{align*}
			m_2^y(1/2,x) = & \int_0^{1/2} \int_{\real^d}2 \alpha(z)(m_1^y(s,z))^2 \rho_1(\frac{1}{2}-s,x,z)dz ds\\ \geq& \int_{1/8}^{1/4} \int_{B_R(y)}2\alpha(z)  \delta_1^2  \delta_2  dz ds\\\geq &  \frac{1}{4}\delta_1^2  \delta_2 \int_{[0,1)^d}\alpha(z)   dz:= \delta_3 >0,
		\end{align*}
		that is, 
		\begin{equation}\label{m2ini}
			m_2^y(1/2,x)\geq \delta_3\chi_{\cube_{y}^d}(x).
		\end{equation}
		Now, comparing the PDEs \eqref{m1eq} and \eqref{m2eq}, and taking into account \eqref{m2ini}, we see that for all $t\geq 0$, $x,y \in \real^d$,
		\begin{equation}\label{m1m2}
			m_2^y(t+1/2,x) \geq \delta_3 m_1^y(t,x).
		\end{equation}
		For a fixed $L>0$, for all $x, y \in \real^d$ with $\frac{\|x-y\|}{t}\leq L$, $t \geq 1/2$, from \thmref{thmmain1}, there exists $c>0$ such that
		\begin{equation}\label{ki}
			m_1^y(t,x)\geq   c  m_1^y(t+1/2,x).
		\end{equation}
		From \eqref{ki} and \eqref{m1m2}, we conclude that there exists a constant $C_L>0$ such that 
		\begin{equation}
			m_2^y(t,x) \geq C_L m_1^y(t,x),
		\end{equation}
		for all $x, y \in \real^d$ with $\frac{\|x-y\|}{t}\leq L$ and $t \geq 1$. In particular, for each $\bv \in \real^d$, we have
		\begin{align}
			\gamma_1(\bv) &=  \lim_{t \to \infty}\frac{\ln m_{1}^{t\bv}(t,x)}{t}  \leq 	\lim_{t \to \infty}\frac{\ln m_{2}^{\bv}(t,x)}{t} = \gamma_2(\bv) 
\end{align}
Because $\EXP(n^{t{\bv}}(t,x)^2)$ is a linear combination of $m_{1}^{t\bv}$ and $m_{2}^{t\bv}$ (by \eqref{kthmoment}) this also implies
\begin{align}
 \lim_{t \to \infty}\frac{\ln \EXP(n^{t{\bv}}(t,x)^2)}{t} = \gamma_2(\bv),    
	\end{align}	
Thus (ii) and (iii) hold for $k = 2$. This completes the basis for induction.

		Next, suppose that (ii) and (iii) hold up to $k-1$ with $k\geq 3$: we will now show that (ii) and (iii) must also hold for $k$, completing the induction. From \eqref{exactmk}, there exists a constant $C_1>0$ such that 
		\[
		m_k^y(t,x)\geq C_1 \int_0^t \int_{\real^d} \alpha(z) m_1^y(s,z)m_{k-1}^y(s,z) \rho_1(t-s,x,z)dz ds =:C_1 I^\ell(t, x,y).
		\]
		Since $\EXP(n^{t{\bv}}(t,x)^k)$ is a convex function of $k$, for each $1 \leq i \leq k-1$,
		\[
		\EXP(n^{t{\bv}}(t,x)^{k-1})\EXP(n^{t{\bv}}(t,x))\geq \EXP(n^{t{\bv}}(t,x)^{k-i})\EXP(n^{t{\bv}}(t,x)^i).
		\]
		Thus, using \eqref{kthmoment}, there exists a constant $C_2 >0$ such that,
		\[
		m_k^y(t,x) \leq C_2  \int_0^t \int_{\real^d}  \alpha(z)\EXP(n^{t{\bv}}(t,x)^{k-1})\EXP(n^{t{\bv}}(t,x))\rho_1(t-s,x,z) dz ds
		\]
		\[
		=: C_2 I^u(t, x,y).
		\]
		In order to prove that the first limit on the right hand side of \eqref{limeq} exists, we need to show that,
		\begin{equation}\label{i1i2}
			\lim\limits_{t \to \infty} \frac{\ln I^\ell(t,x, t\bv)}{t} =  \lim\limits_{t \to \infty} \frac{\ln I^u(t,x, t\bv)}{t}.
		\end{equation}
		We claim that, for all sufficiently large  $M>0$,
		\begin{equation*}
			\gamma_k(\bv) :=  \lim\limits_{t \to \infty}\frac{\ln(I^\ell(t,x,t\bv))}{t} =\lim\limits_{t \to \infty}\frac{\ln(I^u(t,x,t\bv))}{t} = \lim\limits_{t \to \infty}\frac{m_k^{t\bv}(t,x)}{t} 
		\end{equation*}
		\begin{equation}\label{I1}
			= \sup_{\substack{\|w -  \bv\|\leq M\\ u \in (0,1)}} \Big[ u \gamma_{k-1}\Big(\frac{\bv - w}{u}\Big) +  u \gamma_{1}\Big(\frac{\bv - w}{u}
			\Big) + (1-u)\gamma_1\Big(\frac{ w}{1-u}\Big)\Big].
		\end{equation}

		As before, let $0< \eps<\kappa<1$, and partition the domain according to \eqref{domainPartition1}-\eqref{domainPartition2}, and define the integrals
		\[
		I^\ell_j(t,x,t\bv) := \int \!\! \int_{R_j} \alpha(z) m_1^{\bv t}(s,z)m_{k-1}^{\bv t}(s,z)\rho_1(t-s,x,z)dz ds, \quad j=1,\dots,6.
		\] 
so that $I^\ell(t,x,t\bv) = \sum_{j=1}^6 I^\ell_j(t,x,t\bv)$. Using the same arguments as above, it is not  difficult to show that, for each $r>0$, for each $\delta>0$, for all sufficiently small $\eps>0$, for all $\kappa \in (0,1)$ sufficiently close to~$1$,  for all sufficiently large $M$, 
		\[
		\limsup\limits_{t \to \infty} \frac{\ln I^\ell_1(t,x,t\bv)}{t} \leq  -r,
		\]
		\[
		\limsup\limits_{t \to \infty} \frac{\ln I^\ell_5(t,x,t\bv)}{t} \leq  -r,
		\]
		\[
		\limsup\limits_{t \to \infty} \frac{\ln I^\ell_2(t,x,t\bv)}{t} \leq \gamma_1(\bv) + \delta,
		\]
		\[
		\limsup\limits_{t \to \infty} \frac{\ln I^\ell_6(t,x,t\bv)}{t} \leq  \gamma_1(\bv) + \gamma_{k-1}(\bv) + \delta,
		\]
		\[
		\limsup\limits_{t \to \infty} \frac{\ln I^\ell_3(t,x, t\bv)}{t} < -r.
		\]
Now consider integral $I^\ell_4(t,x, t\bv){t}$.  Changing variables $s/t = u$ and $z/t = w$, as before, this is equivalent to
\[
I^\ell_4(t,x, t\bv){t} = t^{ d + 1}\int_{\epsilon }^{\kappa } \int_{\|w - \bv\| \leq M } \alpha(wt) m_1^{\bv t}(ut,wt)m_{k-1}^{\bv t}(ut,wt)\rho_1(t(1 - u),x,wt)dw du.
\] 
The logarithmic asymptotics of $m_{1}$, $m_{k-1}$, and $\rho_1$ are given by \eqref{limeq} and Theorem~\ref{thmmain1}.  Therefore, following Laplace's method, as $t \to \infty$, $\frac{1}{t} \ln I^\ell_4(t,x, t\bv){t}$ is asymptotic to  
\begin{align}
& \frac{1}{t} \ln \left( \int_{\epsilon }^{\kappa } \int_{\|w - \bv\| \leq M } \alpha(wt) m_1^{\bv t}(ut,wt)m_{k-1}^{\bv t}(ut,wt)\rho_1(t(1 - u),x,wt)dw du \right) \no \\
& \quad \sim \frac{1}{t} \ln \left( \int_{\epsilon }^{\kappa } \int_{\|w - \bv\| \leq M } e^{t u \gamma_1( \frac{\bv - w}{u})}e^{t u \gamma_{k-1}( \frac{\bv - w}{u})} e^{t (1 - u) \gamma_1( \frac{ w - (x/t)}{1-u})} dw du \right).
\end{align}
Therefore,
		\[
		\lim_{t \to \infty} \frac{\ln I^\ell_4(t,x, t\bv)}{t}  =  \sup_{\substack{\|w-\bv\|\leq M\\ u \in (\eps,\kappa)}} \Big[ u\gamma_{1}\Big(\frac{\bv - w}{u}
		\Big)+u\gamma_{k-1}\Big(\frac{\bv - w}{u}
		\Big) + (1-u)\gamma_1\Big(\frac{ w}{1-u}\Big)\Big].
		\]
		Combining these estimates, we conclude that
\begin{equation}\label{formgam}
\lim_{t \to \infty} \frac{\ln(I^\ell(t,0,t\bv))}{t} = \sup_{\substack{\|w -  \bv\|\leq M\\ u \in (0,1)}} \Big[ u \gamma_{k-1}\Big(\frac{\bv - w}{u}\Big) +  u \gamma_{1}\Big(\frac{\bv - w}{u}\Big) + (1-u)\gamma_1\Big(\frac{ w}{1-u}\Big)\Big].
		\end{equation}

Now, we justify \eqref{i1i2}, that is, the logarithmic asymptotics of the integrals $I^\ell$ and $I^u$ are equal. The difference between $I^\ell$ and $I^u$ is that $\EXP(n^{t\bv}(t,x)^i)$ in $I^u$ replaces $m_i^{t\bv}(t,x)$ in $I^\ell$. The properties of $m_i^{t\bv}(t,x)$ that were used to derive the asymptotics of $I_1$ included estimate $\eqref{aronsonk}$ and the uniform asymptotics of the logarithm (formula $\eqref{limeq}$). By the inductive assumption, the same uniform asymptotics holds for $\EXP(n^{t\bv}(t,x)^i)$ for $i \leq k-1$. Moreover, by formula \eqref{kthmoment}, the analogue of $\eqref{aronsonk}$ holds for $\EXP(n^{t\bv}(t,x)^i)$. That is, there exist constants $d_i>0$ such that
		\begin{equation}\label{aronsonexpk}
			\EXP(n^{t{\bv}}(t,x)^{i})  \leq d_i \exp\left(d_i t  - \frac{\|y-x\|^2}{d_i(t+1)}\right)
		\end{equation}	 
		for all $(t,x,y) \in \real^+ \times \real^d\times \real^d$, for all $1 \leq i \leq k-1$. Therefore, the logarithmic asymptotics of $I_2$ are the same as that of $I_1$ i.e., \eqref{i1i2} holds. From \eqref{kthmoment},  
		\[
		\liminf\limits_{t \to \infty} \frac{\ln\EXP(n^{t{\bv}}(t,x)^{k})}{t} \geq \lim\limits_{t \to \infty} \frac{\ln m_k^{t\bv}(t,x)}{t}.\]
		From the formula \eqref{I1} which now holds for $k$ and $k-1$ and the inductive hypothesis that $\gamma_{k-1}(v) \geq \gamma_{k-2}(v)$ for each $v \in \real^d$, we observe that that 
		\[\gamma_{k}(\bv) =  \sup\limits_{\|w -  \bv\|\leq M,\\ u \in (0,1)} \Big[ u \gamma_{k-1}\Big(\frac{\bv - w}{u}\Big) +  u \gamma_{1}\Big(\frac{\bv - w}{u}
		\Big) + (1-u)\gamma_1\Big(\frac{ w}{1-u}\Big)\Big] \]\[\geq \sup\limits_{\|w -  \bv\|\leq M,\\ u \in (0,1)} \Big[ u \gamma_{k-2}\Big(\frac{\bv - w}{u}\Big) +  u \gamma_{1}\Big(\frac{\bv - w}{u}
		\Big) + (1-u)\gamma_1\Big(\frac{ w}{1-u}\Big)\Big] = \gamma_{k-1}(\bv).
		\]
		This, along with the inductive hypothesis that $\gamma_{i-1}(\bv) \leq \gamma_i(\bv)$ for each $2 \leq i \leq k-1$, by \eqref{kthmoment}, implies that
		\[
		\limsup\limits_{t \to \infty} \frac{\ln\EXP(n^{t{\bv}}(t,x)^{k})}{t} \leq \lim\limits_{t \to \infty} \frac{\ln m_k^{t\bv}(t,x)}{t}.\] 
		Therefore both the limits in \eqref{limeq} exist and are equal. \\
		
		Form the inductive assumption that $\gamma_i$ is a continuous function for $1 \leq i \leq k-1$,  using formula  \eqref{formgam}, we conclude that $\gamma_k$ is continuous. This concludes the proof of (i)-(iii) through induction.

		We have shown that, for all sufficiently large $M>0$, 
		\[
		\gamma_k(\bv)  = \sup\limits_{\|w -  \bv\|\leq M,\\ u \in (0,1)} \Big[ u \gamma_{k-1}\Big(\frac{\bv - w}{u}\Big) +  u \gamma_{1}\Big(\frac{\bv - w}{u}
		\Big) + (1-u)\gamma_1\Big(\frac{ w}{1-u}\Big)\Big].
		\]
		Therefore, letting $M \to \infty$, we obtain the formula
		\[
		\gamma_k(\bv) = \sup\limits_{w \in \real^d,\\ u \in (0,1)} \Big[ u \gamma_{k-1}\Big(\frac{\bv - w}{u}\Big) +  u \gamma_{1}\Big(\frac{\bv - w}{u}
		\Big) + (1-u)\gamma_1\Big(\frac{ w}{1-u}\Big)\Big], \quad k \in \naturals.
		\]
This completes the proof of (a) in Theorem \ref{thm:nof}.

		\subsection{Proof of part (b)}
	
Using H\"older's inequality, it is easily seen that $\ln\EXP(n^{t{\bv}}(t,x)^k)$ is a convex function of $k$ for each fixed $t \in \real^+, \bv \in \real^d$.  In addition, $\gamma_0 \equiv 0$ and therefore $\gamma_k(\bv) /k$ is a non-decreasing function of $k$, which implies that, if $\gamma_k(\bv) > k\gamma_1(\bv)$, then $\gamma_{k+1}(\bv) > (k+1)\gamma_1(\bv)$. Therefore, $G_{k+1} \subseteq G_k$ must hold for each $k \in \naturals$.  We will complete the proof of Theorem \ref{thm:nof} by showing that there exists a sequence of  constants $\alpha_k > 0$ 
such that $B_{\alpha_k}(0) \subseteq G_k$ and that $\bigcap\limits_{k \in \naturals} G_k = \{0\}$.

		Observe that, for each $k \in \naturals$, $\bv \in \real^d$, $\gamma_k(\bv) \leq k \gamma_1(0)$. To justify this, we use induction. For $k = 1$, the statement is obvious since $\gamma_1$ achieves its maximum at $0$. Now suppose the statement holds up to $k-1$. Then, from the definition of $\gamma_k$,
		\begin{equation}\label{obvio}
			\sup\limits_{w \in \real^d,\\ u \in (0,1)} \Big[ u \gamma_{k-1}\Big(\frac{\bv - w}{u}\Big) +  u \gamma_{1}\Big(\frac{\bv - w}{u}
			\Big) + (1-u)\gamma_1\Big(\frac{ w}{1-u}\Big)\Big] \leq 
		\end{equation} 
		\[\leq \Big[ u (k-1)\gamma_1(0) +  u \gamma_{1}(0) + (1-u)\gamma_1(0)\Big] = k\gamma_1(0).
		\]
		We know that  $-\Phi(0) = \gamma_1(0) = \mu(0)>0$, and $\Phi$ is continuous, therefore, the region $G_1$ is non-empty. 
		 Since the function $\bv \mapsto \gamma_k(\bv)$ is continuous for each $k \geq 1$, 
		 the sets $G_k$ must be closed subsets of $\real^d$.

		Next let us show that each set $G_k$ contains a small ball centered at the origin. As a first step, the following lemma establishes an important property of the functions~$\gamma_k$.
		
		\begin{lemma}\label{gammakdecay}
			For each $k \geq 1$, $\bv \in \real^d$ and $\alpha\in [0,1]$, $\gamma_k(\bv) \leq \gamma_k(\alpha \bv)$.	
		\end{lemma}
		\proof
		We use induction for this proof. For $k = 1$, the statement of the lemma holds since $ \gamma_1(\bv)$ is a twice differentiable strictly concave function and $\bar{\bv} = 0$ is its maximizer.
		
		Suppose the statement of the lemma holds for each $1\leq i \leq k-1$. To show this for $k$, we have, for $0 \leq \alpha \leq 1$,
		\[
		\gamma_k(\alpha\bv)
		= \sup\limits_{w \in \real^d, u \in (0,1)} \Big[ u \gamma_{k-1}\Big(\frac{\alpha\bv - w}{u}\Big) +  u \gamma_{1}\Big(\frac{\alpha\bv - w}{u}
		\Big) + (1-u)\gamma_1\Big(\frac{ w}{1-u}\Big)\Big].
		\]
		Now, substituting $w = \alpha z$, we have
		\[
		\gamma_k(\alpha\bv) =  \sup\limits_{z \in \real^d, u \in (0,1)} \Big[ u \gamma_{k-1}\Big(\frac{\alpha\bv - \alpha z}{u}\Big) +  u \gamma_{1}\Big(\frac{\alpha\bv - \alpha z}{u}
		\Big) + (1-u)\gamma_1\Big(\frac{\alpha z}{1-u}\Big)\Big]
		\]
		\[
		\geq \sup\limits_{z \in \real^d, u \in (0,1)} \Big[ u \gamma_{k-1}\Big(\frac{\bv - z}{u}\Big) +  u \gamma_{1}\Big(\frac{\bv -  z}{u}
		\Big) + (1-u)\gamma_1\Big(\frac{ z}{1-u}\Big)\Big] = \gamma_k(\bv).
		\]
		\qed
		
		Now, in order to prove that each set $G_k = \{\bv \in \real^d\Big| \gamma_k(\bv) = k\gamma_1(\bv), \gamma_1(\bv)\geq 0\}$ contains a small ball centered at the origin, we introduce functions $f_k$ defined below. For each $k \geq 2$, we will first show that there is a small ball centered around the origin on which $f_k(\bv) = k\gamma_1(\bv)\geq 0$. Then we will use induction to show that there is a (smaller) ball centered around the origin on which $f_k(\bv) = \gamma_k(\bv)$.
		
		Let us define, for $k \geq 2$,$$g^{\bv}_k(w,u) := \Big[ku\gamma_1\big(\frac{\bv-w}{u}\big)+(1-u)\gamma_1\big(\frac{w }{1-u}\big)\Big], ~~~ w \in \real^d, ~~ u \in (0,1),$$ 
		and
		\begin{equation}\label{defoff}
			f_k(\bv) :=  \sup\limits_{w \in \real^d, u \in (0,1)} g^{\bv}_k(w,u).
		\end{equation}
		Observe that $f_2 = \gamma_2$. For  $k>2$, the formula for function $f_k$ is similar to the formula of $\gamma_k$, but with $(\gamma_{k-1} + \gamma_1)$ replaced by $k\gamma_1$. For $w = \bv(1-u)$, we have 
		\[
		g^{\bv}_k(\bv(1-u) ,u) = ku\gamma_1(\bv)+(1-u)\gamma_1(\bv) = (1+(k-1)u)\gamma_1(\bv) \to k\gamma_1(\bv) ~~~\text{as} ~~~u \uparrow 1.
		\]
		Therefore, $f_k(\bv) \geq k\gamma_1(\bv)$ for each $\bv \in \real^d$.
		
		The analysis of $g^{\bv}_k(w,u)$ is detailed in the following three lemmas. They show that,  for each $k \geq 2$, there is a small ball centered around the origin $B_{\beta_k}(0)$, such that, for $\bv \in B_{\beta_k}(0)$, the value of the supremum of $g^{\bv}_k(w,u)$ on $\real^d\times (0,1)$ is $k\gamma_1(\bv)$ which, as shown above, can be nearly achieved when $w$ is close to $0$ and $u$ is close to $1$.
		
		The first of the three lemmas, Lemma \ref{small-lem1}, shows that the value of the supremum of $g^{\bv}_k(w,u)$ over the region where $w$ is bounded and $u$ is close to $1$ is $k\gamma_1(\bv)$.
		\begin{lemma}\label{small-lem1} There exist constants  $L>0$, $\delta, \eps_0 > 0$ such that for all $\bv \in B_\delta(0)$, 
			\[
			\sup\{g^{\bv}_k(w, u) \Big| \|w\|\leq L , u \in (1-\eps_0,1)\} = k\gamma_1(\bv).
			\]
		\end{lemma}
		\proof
		We prove the above lemma in 2 steps. In Step I, we show that there exist $\delta_1>0$, $M>0$ and $\eps_1>0$ such that, for each $\bv \in B_{\delta_1}(0)$, for each $(w,u) = (\ell \eps, 1 - \eps)$, with $L/\eps >\|\ell\|>M$, and $\eps\in (0, \eps_1)$, we have $g^{\bv}_k(\ell \eps, 1 - \eps) <   k\gamma_1(\bv)$.
		
		In Step II, we show that there exist constants $\delta<\delta_1, \eps_0 \in  (0, \eps_1)$ such that for all $\bv \in B_{\delta}(0)$, for all $\|\ell\|\leq M$, $$\frac{d}{d\eps} \Big[g^{\bv}_k(\ell \eps, 1 - \eps)\Big]  \leq 0$$ for all $\eps < \eps_0$.

		{\bf{Step I:}} Note that from Lemma \ref{gammakdecay},
		\[
		g^{\bv}_k(\ell \eps, 1 - \eps) = k(1 - \eps) \gamma_1\big(\frac{\bv- \ell \eps}{1 - \eps}\big) + \eps\gamma_1(\ell)	\leq k\gamma_1\big(\bv- \ell \eps\big) +  \eps\gamma_1(\ell),
		\]
		where $L>0$ and $\delta_1>0$ are such that $\gamma_1(\bv - \ell\eps) >0$ for all  $\bv \in B_{\delta_1}(0)$, and $\|\ell\eps\|\leq L$. Thus, in order to prove that $ g^{\bv}_k(\ell \eps, 1 - \eps) <   k\gamma_1(\bv)$, it is enough to show that
		\[
		\frac{\eps}{k}\gamma_1(\ell) \leq  \gamma_1(\bv) - \gamma_1\big(\bv- \ell \eps\big).
		\]	
		From \eqref{aronsonm1}, we know that, for all $v\in \real^d$,
		\[
		\gamma_1(v) \leq a_1 - \frac{\|v\|^2}{a_1}.
		\] 
		Therefore,  we only need to show that 
		\[
		\frac{\eps}{k}(	a_1-\frac{\|\ell\|^2}{a_1}) \leq  \gamma_1(\bv) - \gamma_1\big(\bv- \ell \eps\big).
		\]
		Using Taylor's formula, we have
		\[
		\gamma_1(\bv) - \gamma_1\big(\bv- \ell \eps\big) = \eps\langle \nabla\gamma_1(\bv),\ell\rangle - \frac{\eps^2}{2}\langle D^2\gamma_1(\bv - q\ell\eps)\ell,\ell\rangle
		\]
		for some $q \in (0,1)$. Thus, we need to show that 
		\[
		\frac{1}{k}(	a_1-\frac{\|\ell\|^2}{a_1}) \leq \langle \nabla\gamma_1(\bv),\ell\rangle - \frac{\eps}{2}\langle D^2\gamma_1(\bv - q\ell\eps)\ell,\ell\rangle.
		\]
		That is, we need to show that
		
		\begin{equation}\label{opop}
			\frac{1}{k}(	a_1-\frac{\|\ell\|^2}{a_1}) + \frac{\eps}{2}\langle D^2\gamma_1(\bv - q\ell\eps)\ell,\ell\rangle -\langle \nabla\gamma_1(\bv),\ell\rangle\leq 0.
		\end{equation}
		Let $\bv\in B_{\delta_1}(0)$. Let $C =  \sup\{\|\partial_{i}\gamma_1(\bv)\|\Big|\bv \in B_{\delta_1}(0)\}$. Then we have the following lower bound, \[
		\langle \nabla\gamma_1(\bv),\ell\rangle \geq -C\|\ell\|.
		\]
		Let us fix $M\geq 1$ such that the following quadratic expression is positive, that is, 
		\[ 
		\frac{x^2 }{2ka_1}  -Cx -\frac{a_1}{k}\geq 0~~~\text{ for all}~~ \|x\|\geq M. 
		\] 
		For each $\bv \in B_{\delta_1}(0)$, $\|\ell \eps\|  \leq L$, $q\in (0,1)$ we have $(\bv -\eps q\ell )\in B_{\delta_1 + L}(0)$. Let  
		\[
		R = \sup\Big\{\|\partial_{i,j}\gamma_1(v)\|\Big| v \in B_{(\delta_1+L)}(0), 1\leq i,j\leq d\Big\}.\] This is a finite constant since the function $\gamma_1$ is twice continuously differentiable. Choose $ \eps_1>0$ such that $\eps_1R<\frac{1}{2a_1k}$. Then, for all $L/\eps \geq \|\ell\| \geq M$ and $\eps< \eps_1$, $\bv \in B_{\delta_1}(0)$,
		\[
		\frac{\|\ell\|^2}{a_1k} + \langle \nabla\gamma_1(\bv),\ell\rangle - \frac{a_1}{k} - \frac{\eps}{2}\langle D^2\gamma_1(\bv - q\ell\eps)\ell,\ell\rangle \geq \]
		\[\geq \frac{\|\ell\|^2}{2a_1k} - \frac{\eps}{2}\langle D^2\gamma_1(\bv - q\ell\eps)\ell,\ell\rangle\geq \eps_1 R\|\ell\|^2 -\frac{\eps}{2}\langle D^2\gamma_1(\bv - q\ell\eps)\ell,\ell\rangle \geq 0,
		\]
		which proves \eqref{opop}.

		{\bf{Step II:}} Recall that
		\[
		g^{\bv}_k(\ell \eps, 1 - \eps)= k(1 - \eps) \gamma_1\big(\frac{\bv- \ell \eps}{1 - \eps}\big) + \eps\gamma_1(\ell).
		\]
		Differentiating with respect to $\eps$ we obtain,  \[\frac{d}{d\eps}\Big[	g^{\bv}_k(\ell \eps, 1 - \eps)\Big] = -k\gamma_1\big(\frac{\bv- \ell \eps}{1 - \eps}\big) + \frac{k(\bv - \ell)}{1 - \eps}\nabla\gamma_1\big(\frac{\bv- \ell \eps}{1 - \eps}\big) + \gamma_1(\ell).\]
		Using the fact that the maximum of the function $\gamma_1(\bv)$ is achieved at  $\bv = 0$ and the fact that $\gamma_1(\bv)$ is strictly concave, choose $\delta_2 \in (0, \delta_1)$ be such that 
		\[
		\min\{\gamma_1(\bv)\Big| \bv \in B_{\delta_2}(0)\} > \frac{7}{4k}\gamma_1(0).
		\]
		Let $\eps_2\in (0,\eps_1)$ be such that for each $\bv \in B_{\frac{\delta_2}{2}}(0)$, for all $\eps < \eps_2$, and $\|\ell\| \leq M$, the vector  $\big(\frac{\bv- \ell \eps}{1 - \eps}\big)$ belongs to the $ B_{\delta_2}(0)$.
		
		Now choose $\delta\in (0, \delta_2/2)$ such that, for each $\bv \in B_{\delta}(0)$, we have \[
		k\langle(\bv - l),\nabla\gamma_1(\bv) \rangle< \frac{1}{8}\gamma_1(0),
		\]  for all $\|\ell\| \leq M$.
		This is possible since $\gamma_1$ achieves its maximum at $0$, that is $\nabla\gamma_1(0) = 0$. Choose $ \eps_0>0 $ with $\eps_0< (0,\eps_2)$ such that for all $\bv \in B_{\delta}(0)$ and for all $\|\ell\| \leq M$, we have 
		\[
		k\langle\frac{(\bv - \ell )}{1 - \eps_0} ,\nabla\gamma_1\big(\frac{\bv- \ell \eps_0}{1 - \eps_0}\big)\rangle < \frac{1}{4}\gamma_1(0).
		\]
		Thus, for all $\bv \in B_{\delta}(0)$, for all $\eps< \eps_0$ and for all $\|\ell\| \leq M$,
		\[
		\frac{d}{d\eps}\Big[	g^{\bv}_k(\ell \eps, 1 - \eps)\Big]< -\frac{7}{4}\gamma_1(0) + \frac{1}{4}\gamma_1(0) + \gamma_1(0) = -\frac{1}{2}\gamma_1(0) < 0.
		\]
		Thus, we conclude that, for all $\bv \in B_{\delta}(0)$,
		\[
		\sup\{g^{\bv}_k(w, u) \Big| \|w\|\leq L , u \in (1-\eps_0,1)\} \leq  k\gamma_1(\bv).
		\]
		But we know that, if $(w,u)  = (\bv(1-u), u)$ and $u$ approaches $1$, the value of $g^{\bv}_k(w, u)$ approaches $k\gamma_1(\bv)$. Therefore, 
		\[
		\sup\{g^{\bv}_k(w, u) \Big| \|w\|\leq L , u \in (1-\eps_0,1)\} = k\gamma_1(\bv).
		\]\qed
		
		The next lemma shows that the supremum of $g$ cannot be achieved if $u$ is close to 1, and  $w$ is separated from  the origin. 
		\begin{lemma}\label{small-lem2}
			For each $L>0$, there exist $\delta >0$ and $\eps_0 >0$ such that, for all $ \bv \in B_{\delta}(0)$, 
			\[\sup\{ g^{\bv}_k(w, u)\Big|  1 - \eps_0 < u \leq 1 , \|w\| \geq L \} < k\gamma_1(\bv).\]
		\end{lemma}
		\proof Note that, 
		\[\|\frac{\bv - w}{u}\|
		\geq \frac{L}{2} >0 ~~\text{for each}~~ \|w\|\geq L, ~~\|\bv\|\leq L/2,~~ u \in [1/2,1).\]
		Take $\alpha \in (0,1)$ such that $\gamma_1(\ell) \leq \alpha \gamma_1(0)$ for all $\|\ell\| \geq L/2$. Here, we used the fact that the maximum of the function $\gamma_1$ is achieved at $\bv =0$ and $\gamma_1(\bv)$ is continuous.
		
		Choose an $\eps_0 < 1/2$ such that $\alpha + \frac{\eps_0}{k} <1$. Thus,
		\[
		g^{\bv}_k(w, u) = ku\gamma_1\big(\frac{\bv-w}{u}\big)+(1-u)\gamma_1\big(\frac{w }{1-u}\big) \leq (k\alpha +\eps_0)\gamma_1(0),
		\]
		for all $\bv \in B_{L/2}(0)$, $\| w \|> L, u \in [1-\eps_0,1)$.
		Now we choose a $\delta>0$ with $\delta < L/2$ such that, for all $\bv \in B_\delta(0)$, we have 
		$$\gamma_1(\bv) > (\alpha +\frac{\eps_0}{k})\gamma_1(0).$$ We can choose such a $\delta >0$ since $1 > (\alpha +\frac{\eps_0}{2})>0$, the maximum of the function $\gamma_1(\bv)$ is achieved at  $\bv = 0$, and $\gamma_1(\bv)$ is continuous.
		
		Thus, for all $ \bv \in B_{\delta}(0)$, we have
		\[\sup\{ g^{\bv}_k(w, u)\Big|  1 - \eps_0 < u \leq 1 , \|w\| \geq L \} < k\gamma_1(\bv).
		\]
		\qed
		
		The last of the three lemmas, Lemma \ref{small-lem3}, shows that there is a small ball centered around the origin, on which the value of the supremum of $g^{\bv}_k(w,u)$ in the region where $w\in \real^d$ and  $u$ is away from $1$ is strictly less than $k\gamma_1(\bv)$.
		\begin{lemma}\label{small-lem3}
			For each  $\eps>0$, there exists a $\delta >0$ such that, for all $\bv \in B_\delta(0)$,
			\[
			\sup\{ g^{\bv}_k(w, u)\Big|  1 - \eps > u \geq 0 ,  w \in \real^d\} < k\gamma_1(\bv).
			\]
		\end{lemma}
		\proof
		Choose $\delta>0$ such that for all $ \bv \in B_\delta(0)$, $\gamma_1(\bv) > (1 - \frac{(k-1)\eps}{k})\gamma_1(0)$.
		We can choose such a $\delta >0$ since $1 > (1 - \frac{(k-1)\eps}{k})>0$, the maximum of the function $\gamma_1$ is achieved at $\bv =0$ and $\gamma_1(\bv)$ is continuous.
		Then, for all $\bv \in B_\delta(0)$, for all $w \in \real^d$ and $u \in [0, 1-\eps)$, 
		\begin{align*}
			&g^{\bv}_k(w, s)  = ku\gamma_1\big(\frac{\bv-w}{u}\big)+(1-u)\gamma_1\big(\frac{w }{1-u}\big)\\ & \leq ku \gamma_1(0)  + (1-u) \gamma_1(0) = (1+(k-1)u) \gamma_1(0)\\& \leq (k - (k-1)\eps)\gamma_1(0) < k\gamma_1(\bv).
		\end{align*}
		Therefore,  for all $ \bv \in B_\delta(0)$,
		\[
		\sup\{ g^{\bv}_k(w, u)\Big|  1 - \eps > u \geq 0 ,  w \in \real^d\}< k\gamma_1(\bv).
		\]
		\qed
		
		Thus, by the above three lemmas, there exists a sequence of positive constants $\{\beta_k\}_{k \geq 1}$ such that, for all $\bv \in B_{\beta_k}(0)$, \begin{equation}\label{fksup}
			f_k(\bv)= \lim\limits_{u \uparrow 1}g^{\bv}_k(\bv(1-u) ,u) = k\gamma_1(\bv). \end{equation}
		Now let us show that there exists a sequence of positive constants $\{\alpha_k\}_{k \geq 1}$ such that, for all $\bv \in B_{\alpha_k}(0)$, $f_k(\bv) = \gamma_k(\bv)$. This will be proved by induction.
		
		For $k = 2$, by the definition of $\gamma_2$, we have that, $f_2(\bv) = \gamma_2(\bv)$ for each $\bv \in\real^d$. Now suppose there exists constants $\alpha_i$ for $1 \leq i \leq k-1$ with $\alpha_i \in (0,\beta_i]$ such that $\gamma_i (\bv) = f_i(\bv) = i\gamma_1(\bv)$ for all $\bv \in B_{\alpha_i}(0)$. We need to show that there exists $\alpha_k \in (0,\beta_k]$ such that, for all $\bv \in B_{\alpha_k}(0)$, we have
		\[
		\gamma_k(\bv) := \sup\limits_{w \in \real^d, u \in (0,1)} \Big[ u \gamma_{k-1}\Big(\frac{\bv - w}{u}\Big) +  u \gamma_{1}\Big(\frac{\bv - w}{u}
		\Big) + (1-u)\gamma_1\Big(\frac{ w}{1-u}\Big)\Big]
		\]
		\[
		=  \sup\limits_{w \in \real^d, u \in (0,1)} \Big[ ku \gamma_1\Big(\frac{\bv - w}{u}\Big) + (1-u)\gamma_1\Big(\frac{ w}{1-u}\Big)\Big] =: f_k(\bv).
		\]
		To show this, it is enough to show that the supremum in the definition of $\gamma_k$ is achieved in the part of the space where the values of $\gamma_{k-1}$ and $(k-1)\gamma_1$ coincide. Let us define the cone   $\Gamma_k(\bv) =  \{(w, u) \in \real^d\times(0,1): \frac{|\bv - w|}{u} \leq \alpha_{k-1}\} \subseteq \real^d\times(0,1)$. It remains to show that $\Big[ u \gamma_{k-1}\Big(\frac{\bv - w}{u}\Big) +  u \gamma_{1}\Big(\frac{\bv - w}{u}
		\Big) + (1-u)\gamma_1\Big(\frac{ w}{1-u}\Big)\Big]$ on the set $\Gamma_k(\bv) ^c :=\real^d\times(0,1) \setminus \Gamma_k(\bv)$ is dominated by the supremum of the same expression over the set $\Gamma_k(\bv) $. We will show that there exists $\alpha_k>0$ such that, for all $\bv \in B_{\alpha_k}(0)$, we have
		\[
		\sup\limits_{(w,u) \in \Gamma_k(\bv) ^c} \Big[ u \gamma_{k-1}\Big(\frac{\bv - w}{u}\Big) +  u \gamma_{1}\Big(\frac{\bv - w}{u}
		\Big) + (1-u)\gamma_1\Big(\frac{ w}{1-u}\Big)\Big] \]\begin{equation}\label{yu}
			\leq \sup\limits_{(w,u)  \in  \Gamma_k(\bv)}\Big[ u \gamma_{k-1}\Big(\frac{\bv - w}{u}\Big) +  u \gamma_{1}\Big(\frac{\bv - w}{u}
			\Big) + (1-u)\gamma_1\Big(\frac{ w}{1-u}\Big)\Big].
		\end{equation}
		Note that, for each $(w,u)  \in  \Gamma_k(\bv)$, the expression on the RHS is 
		\[
		\sup\limits_{(w,u)  \in  \Gamma_k(\bv)}\Big[ ku \gamma_{1}\Big(\frac{\bv - w}{u}
		\Big) + (1-u)\gamma_1\Big(\frac{ w}{1-u}\Big)\Big].
		\]This expression,  as follows from \eqref{fksup}, is equal to $f_k(\bv) = k\gamma_1(\bv)$, as long as $(w,u) = (\bv(1-u),u) \in \Gamma_k(\bv)$ and $\|\bv\| \leq \beta_k$. That is, $\|\bv\|\leq \min\{\alpha_{k-1}, \beta_{k}\}$. The inequality \eqref{yu} is justified by the following lemma.
		\begin{lemma}There exists $0<\alpha_k< \min\{\alpha_{k-1},\beta_k\}$ such that, for each $\bv \in B_{\alpha_k}(0)$,
			\begin{equation}\label{okn}
				u \gamma_{k-1}\Big(\frac{\bv - w}{u}\Big) +  u \gamma_{1}\Big(\frac{\bv - w}{u}
				\Big) + (1-u)\gamma_1\Big(\frac{ w}{1-u}\Big) < k\gamma_1(\bv),
			\end{equation}
			for all $(w, u) \in \Gamma_k(\bv) ^c$.
			
		\end{lemma}
		\begin{proof} The lemma will be proved in 2 steps. In Step I, the part of set $\Gamma_K(\bv)^c$ where $u$ is close to $1$ is considered. In this part of the set, we make use of the fact that $w$ is bounded from below.
			
			In Step II, the part of set $\Gamma_k(\bv)^c$ where $u$ is away from $1$ is considered. In this part of the set, the left hand side of $\eqref{okn}$ can be made strictly smaller than $k\gamma_1(0)$, while the right hand side can be made arbitrarily close to $k\gamma_1(0)$ by choosing $\bv$ in a small enough ball around the origin. 

			{\bf Step I: } Let $\delta_1 = \min\{\alpha_{k-1},\beta_k\}/4$. For all $(w,u) \in \Gamma_k(\bv)^c$, $\|\bv\| \leq \delta_1$ and $u \in [3/4,1)$, we have, 
			\[
			4\delta_1\leq \alpha_{k-1} <  \|\frac{\bv-w}{u }\| \leq \frac{4}{3}\|\bv - w\| \leq \frac{4}{3} \big(\|\bv\| + \|w\|\big) \leq \frac{4}{3}\big(\delta_1 + \|w\|\big).
			\]
			Therefore, $\|w\| \geq 2\delta_1$. Using $\eqref{aronsonm1}$, there exist $a>0$ and $M>0$ such that, for all $\|\ell\| \geq M$,
			\begin{equation}\label{mun}
				\gamma_1(\ell) < -a\|\ell\|^2.
			\end{equation}
			In addition, we choose $M>0$ large enough such that $\delta_1/M < 1/4$. Observe that, from $\eqref{mun}$, 
			\[(1-u)\gamma_1\Big(\frac{w}{1-u}\Big) \leq -a\frac{\|w\|^2}{1-u} \leq -\frac{4a\delta_1^2}{1-u},
			\]
			for all $u \in [1-\frac{\delta_1}{M}, 1)$, $(w,u) \in \Gamma_k(\bv)^c$ if $\|\bv\| \leq \delta_1$.
			From \eqref{obvio}, for each $(\bv, w,u) \in \real^d \times\real^d\times(0,1)$, we know that 
			\[
			u \gamma_{k-1}\Big(\frac{\bv - w}{u}\Big) +  u \gamma_{1}\Big(\frac{\bv - w}{u}
			\Big)\leq uk\gamma_1(0).
			\]
			Choosing  $\eps \in (0, \delta_1/M)$ such that $k\gamma_1(0) < 4a\delta_1^2/2\eps$, we obtain that, for each $(w, u)\in \real^d\times(1-\eps,1) \setminus\Gamma_k(\bv)$ for $\|\bv\| \leq \delta_1$, the left-hand side of equation \eqref{okn} is negative. 
			
			We now choose $\delta_2 \in (0, \delta)$ such that, for all $\|\bv\|< \delta_2$, we have $k\gamma_1(\bv)>0$. Thus the inequality $\eqref{okn}$ holds for all $\|\bv\|< \delta_2$,  for each $(w, u)\in \real^d\times(1-~\eps,1~)~\setminus~\Gamma_k(\bv)$. 

			{\bf Step II:}  Let $\eps>0$ be fixed. Choose $\alpha_k \in (0, \delta_2)$ such that $\gamma_1(\bv) > (1 - \eps + \frac{\eps}{k})\gamma_1(0)$ for all $|\bv|< \alpha_k$. Using \lemref{gammakdecay}, for each $\ell \in \real^d$, $u \in (0,1), i \in \naturals$, we have $\gamma_i(\ell)< \gamma_i(u\ell)$. Therefore, 
			\[
			u\gamma_{k-1}\Big(\frac{\bv - w}{u}\Big) +  u \gamma_{1}\Big(\frac{\bv - w}{u}
			\Big) < u\gamma_{k-1}(\bv - w) +  u \gamma_{1}(\bv - w) \leq (k-1)u\gamma_1(0) + u \gamma_{1}(0),
			\]
			where the last inequality follows from the trivial observation that $\gamma_{i}(\ell) \leq i\gamma_1(0)$, for all $i \in \naturals, \ell \in \real^d$.
			Therefore, for $u \in (0,1-\eps)$, the left hand side of \eqref{okn} can be bounded above as follows, 
			\[
			u \gamma_{k-1}\Big(\frac{\bv - w}{u}\Big) +  u \gamma_{1}\Big(\frac{\bv - w}{u}
			\Big) + (1-u)\gamma_1\Big(\frac{ w}{1-u}\Big) \leq (k-1)u\gamma_1(0) + u\gamma_1(0) + (1-u)\gamma_1(0) 
			\]
			\[
			\leq k(1-\eps + \frac{\eps}{k})\gamma_1(0).
			\]
			From the definition of $\alpha_k$, for all $|\bv|< \alpha_k$, $k\gamma_1(\bv) > k(1-\eps + \frac{\eps}{k})\gamma_1(0)$. Thus, we have shown that inequality $\eqref{okn}$ holds for all $|\bv|< \alpha_k$, for each $(w, u)\in \real^d \times(0, 1-\eps)\setminus\Gamma_k(\bv)$. 
		\end{proof}
		
		\noindent Now we prove that $\bigcap_{k \geq 1} G_k = \{0\}$. Let $\bv \in G_1$ be fixed, with $\|\bv\|> 0$. Now, we show that there exists $k \in \naturals$, large enough, such that $\gamma_k(\bv) > k\gamma_1(\bv)$. That is, there exists a pair $(w,u) \in \real^d \times(0,1)$ such that 
		\[
		u \gamma_{k-1}\Big(\frac{\bv - w}{u}\Big) +  u \gamma_{1}\Big(\frac{\bv - w}{u}
		\Big) + (1-u)\gamma_1\Big(\frac{ w}{1-u}\Big) > k\gamma_1(\bv).
		\]
		We first pick $w = \bv$. Then we need to show that there exist $u \in (0,1)$ and $k \in \naturals$, such that
		\begin{equation}\label{olp}
			u\gamma_1(0) + \frac{1-u}{k} \gamma_1\Big(\frac{\bv}{1-u}\Big) > \gamma_1(\bv).
		\end{equation}
		Let  $u = 1-\eps$ where $\eps>0$ small enough such that  $(1-\eps)\gamma_1(0) > \gamma_1(\bv)$. This is possible because $\|\bv\| >0$ and $\gamma_1(\bv)$ achieves its maximum value at $\bv = 0$.  Define $\eta  = (1-\eps)\gamma_1(0) - \gamma_1(\bv)>0$. Keeping $\bv$ and $\eps$ fixed, we pick $k \in \naturals$ large enough such that,
		\[
		\Big | \frac{\eps}{k} \gamma_1\Big(\frac{\bv}{\eps}\Big) \Big| < \eta /2.
		\]
		Therefore, 
		\[
		(1-\eps)\gamma_1(0) + \frac{\eps}{k} \gamma_1\Big(\frac{\bv}{\eps}\Big)   = \eta +  \gamma_1(\bv) +\frac{\eps}{k} \gamma_1\Big(\frac{\bv}{\eps}\Big)  >  \gamma_1(\bv) + \frac{\eta}{2} > \gamma_1(\bv).
		\]
		Thus $\bigcap_{k \geq 1} G_k = \{0\}$. This concludes the proof of Theorem \ref{thm:nof}.

		\section{Proof of \thmref{tot}}
		
		Without loss of generality, we may assume that $\bar{\bv} = 0$, which simplifies our notation.	Observe that the functions $f_k$ are clearly positive and continuous on the $d$-dimensional cube $\cube_0^d = [0,1)^d$, from their recursive definition. As in \eqref{kthmoment}, for each $k \geq 1$,  
		\begin{equation}\label{onn}
			\EXP(N(t,x)^k)	= \sum_{i = 1}^{k}S(k,i)\bar{m}_i(t,x),
		\end{equation}
		where $$\bm_i(t,x) = \int_{\real^d}....\int_{\real^d} \rho_k(t,x,y_1,y_2,...,y_k) dy_1...dy_k,$$
		where $\rho_i$'s are the particle density and higher order correlation functions, as defined in \eqref{rho1} and \eqref{rhon}. Thus, we observe that $\bm_i(t,x)$ satisfy the following PDEs on $\TOR^d$:
		\begin{equation}\label{barm1}
			\partial_t \bm_1(t,x) = \cL_x\bm_1(t,x), ~~~~\bm_1(0,x) \equiv 1,
		\end{equation}
		while, for $k \geq 2$,
		\begin{equation}\label{barmk}
			\partial_t \bm_k(t,x) = \cL_x\bm_k(t,x) + \alpha(x) \sum_{i = 1}^{k-1}\beta_i^k\bm_i(t,x)\bm_{k-i}(t,x), ~~~~\bm_k(0,x) \equiv 0,
		\end{equation}
		where $\beta_i^k = k!/(i!(k-i)!)$.
		
		We will prove the following lemma after completing the proof of the theorem. 
		\begin{lemma}\label{formmk}
			For each $k\in \naturals$, $x \in [0,1)^d$, 
			\begin{equation}\label{obb}
				\bm_k(t,x) = e^{k\mu t}\Big[f_k(x) + q_k(t,x)\Big],
			\end{equation}
			where $\lim\limits_{t \to \infty}q_k(t,x) = 0$ uniformly in $x \in [0,1)^d$, and $f_k$ have been defined in \eqref{fk}.
		\end{lemma}
		\noindent Using formula \eqref{obb} in \eqref{onn},  we get
		\[
		\frac{\EXP(N(t,x)^k)}{e^{k\mu t}}	= \sum_{i = 1}^{k}S(k,i)e^{-(k-i)\mu t }\Big[f_i(x) + q_i(t,x)\Big].
		\]
		Therefore,
		\begin{align*}
			\lim\limits_{t \to \infty} \frac{\EXP(N(t,x)^k)}{e^{k\mu t}}& = f_k(x) + \lim\limits_{t \to \infty}\Big( q_k(t,x) + \sum_{i = 1}^{k-1}S(k,i)e^{-(k-i)\mu t }\Big[f_i(x) + q_i(t,x)\Big]\Big)\\& =  f_k(x).
		\end{align*}
		Now, we use induction to show that there exists a constant $A>0$ such that, for every $x \in [0,1)^d$,  $f_k(x) \leq A^k k!$. For $k = 1$, we know that the eigenfunction $\ph(x)$ corresponding to the principle eigenvalue $\mu$ of the operator $\cL$  on the $d$-dimensional torus $\TOR^d$ is a positive and continuous function. Therefore, there exists a constant $A_1>1$ such that, for every $x \in (0,1]^d$, $\ph(x) \leq A_1$.
		
		Suppose that for all $1 \leq j \leq k-1$, $x \in [0,1)^d$, $f_j(x) \leq A_1^j j!$. Then, from the definition of the function $f_k$, we get
		\[
		f_k(x) = \sum_{i =1}^{k-1}\beta_i^k \int_0^\infty\int_{[0,1)^d}e^{-k\mu t}\alpha(z)f_i(z)f_{k-i}(z)\rh(t,x,z)dzdt\]\[\leq  A_1^k(k-1) k!\int_0^\infty\int_{[0,1)^d}e^{-k\mu t}\alpha(z)\rh(t,x,z)dzdt.
		\]
		Recall that the operator $\cL-\mu$ has principle eigenvalue zero, while the principle eigenfunction of the adjoint operator $(\cL- \mu)^*$ is $\ph^*$ (with $\int_{[0,1)^d}\ph^*(z)\,dz = 1$). Therefore, there exists a constant $C>0$ such that, for every $x \in [0,1)^d$, $t>0$,  \[\int_{[0,1)^d}e^{ -t \mu }\alpha(z) \rh(t,x,z) dz \leq \big(\sup\limits_{x\in [0,1)^d}\alpha(x)\big)\int_{[0,1)^d}e^{ -t \mu }\rh(t,x,z) dz \leq C.\] Therefore, 
		\[
		f_k(x) \leq   CA_1^k (k-1) k!\int_0^\infty e^{-(k-1)\mu t}dt \leq k!A_1^k C/\mu.
		\]
		If $C/\mu\leq 1$, we pick $A = A_1$, and if $C/\mu > 1$, choose $A = A_1C/\mu$. With this choice of $A$ we obtain that, for every $x \in [0,1)^d$,  $f_k(x) \leq A^k k!$. From the convergence of all the moments of $N(t,x)/e^{\mu t}$, it follows that, there exists a random variable $\xi_x$ with the moments $f_k(x)$ (see \cite{frechet1931}). The uniqueness of the distribution of $\xi_x$ follows from the bound on $f_k$ by the Carleman theorem.  Except for a proof of \lemref{formmk}, this concludes the proof of Theorem~\ref{tot}.

		\proof[Proof of \lemref{formmk}] We use induction to prove this lemma. 
		The principle eigenvalue of the operator $\cL$ is $\mu>0$, and the corresponding eigenfunction $\ph(x)>0$. Thus, from the theory of elliptic operators, from \eqref{barm1}, there exists a function $q_1(t,x)$ such that 
		\begin{equation*}
			\bm_1(t,x) = e^{\mu t}\big[\ph(x) + q_1(t,x)\big],
		\end{equation*}
		where
		\[
		\lim\limits_{t \to \infty}q_1(t,x) = 0
		\]
		uniformly in $x \in [0,1)^d$. This gives \eqref{obb} for $k= 1$ with $f_1(x) = \ph(x)$. Suppose  that the conclusion of the lemma holds up to $k-1$, where $k \geq 2$. From \eqref{barmk}, using Duhamel's formula, we get
		\[
		\bm_k(t,x) = \int_0^t\int_{[0,1)^d}  \alpha(z) \sum_{i = 1}^{k-1}\beta_i^k\bm_i(s,z)\bm_{k-i}(s,z) \rh(t-s,x,z) dz ds.
		\]
		By the inductive assumption,
		\begin{align*}
			\bm_k(t,x) & =  \sum_{i = 1}^{k-1} \beta_i^k \int_0^t\int_{[0,1)^d}  e^{ k\mu s}\alpha(z)f_i(z)f_{k-i}(z) \rh(t-s,x,z) dz ds \\& + \sum_{i = 1}^{k-1} \beta_i^k \int_0^t\int_{[0,1)^d}  e^{ k\mu s}\alpha(z)h_i(s,z) \rh(t-s,x,z) dz ds,
		\end{align*}
		where \[h_i(s,z) := q_i(s,z)q_{k-i}(s,z) + q_i(s,z)f_{k-i}(z) + q_{k-i}(s,z)f_{i}(z).
		\]  
		After the change of variables $u  = t-s$, we get
		\begin{align*}
			\bm_k(t,x) & =  e^{k\mu t}\sum_{i = 1}^{k-1} \beta_i^k \int_0^t\int_{[0,1)^d}  e^{ -u k\mu }\alpha(z)f_i(z)f_{k-i}(z) \rh(u,x,z) dz du \\& + e^{k\mu t} \sum_{i = 1}^{k-1} \beta_i^k \int_0^t\int_{[0,1)^d}  e^{-u k\mu}\alpha(z)h_i(t-u,z) \rh(u,x,z) dz du\\
			& = e^{k\mu t} f_k(x) -  e^{k\mu t}\sum_{i = 1}^{k-1} \beta_i^k \int_t^\infty \int_{[0,1)^d}  e^{ -u k\mu }\alpha(z)f_i(z)f_{k-i}(z) \rh(u,x,z) dz du  \\&+  e^{k\mu t} \sum_{i = 1}^{k-1} \beta_i^k \int_0^t\int_{[0,1)^d}  e^{-u k\mu}\alpha(z)h_i(t-u,z) \rh(u,x,z) dz du.
		\end{align*}
		Define
		\begin{align*}
			q_k(t,x) &:=  \sum_{i = 1}^{k-1}\beta_i^k\Big(  \int_0^t\int_{[0,1)^d}  e^{-u k\mu}\alpha(z)h_i(t-u,z) \rh(u,x,z) dz du \\& - \int_t^\infty \int_{[0,1)^d}  e^{ -u k\mu }\alpha(z)f_i(z)f_{k-i}(z) \rh(u,x,z) dz du \Big).
		\end{align*}
		Thus, we have  
		\[
		\bm_k(t,x) = e^{k\mu t}\Big[f_k(x) + q_k(t,x)\Big].
		\]
		It remains to show that $\lim\limits_{t \to\infty}q_k(t,x) = 0$ uniformly in $x \in [0,1)^d$. Since the functions $\tc, f_i, f_{k-i}$ are non-negative and continuous on $[0,1)^d$, there exists a constant $C_i>0$ such that $0 \leq \alpha(z)f_i(z)f_{k-i}(z)<C_i$ for all $z\in[0,1)^d$. Therefore, 
		\begin{equation}\label{kok}
			\int_t^\infty \int_{[0,1)^d}  e^{ -u k\mu }\alpha(z)f_i(z)f_{k-i}(z) \rh(u,x,z) dz du \leq C_i\int_t^\infty \int_{[0,1)^d}e^{ -u k\mu } \rh(u,x,z) dz du .
		\end{equation}
		Therefore, from \eqref{rhobound}, the right hand side of the \eqref{kok} goes to zero uniformly in $x\in [0,1)^d$.
		To deal with the sum in the definition of $q_k(t,x)$, we break up the integral in two parts as follows,
		\[
		\Big|\int_0^t\int_{[0,1)^d}  e^{-u k\mu}\alpha(z)h_i(t-u,z) \rh(u,x,z) dz du\Big| \]\[\leq \Big|\int_0^{t/2}\int_{[0,1)^d}  e^{-u k\mu}\alpha(z)h_i(t-u,z) \rh(u,x,z) dz du\Big|\]\[+ \Big|\int_{t/2}^t\int_{[0,1)^d}  e^{-u k\mu}\alpha(z)h_i(t-u,z) \rh(u,x,z) dz du\Big|\]
		\[ \leq \sup\limits_{s \in (t/2,t), x \in [0,1)^d}|\alpha(x)h_i(s,x)|\int_0^{t/2}\int_{[0,1)^d}  e^{-u k\mu} \rh(u,x,z) dz du\]
		\[
		+ \sup\limits_{s \in (0 ,t/2), x \in [0,1)^d}|\alpha(x)h_i(s,x)| \int_{t/2}^t\int_{[0,1)^d}  e^{-u k\mu} \rh(u,x,z) dz du
		\]
		From $\eqref{rhobound}$, the integral in the first term is bounded and from the inductive hypothesis,
		\[
		\lim\limits_{t \to \infty}\sup\limits_{s \in (t/2,t), x \in [0,1)^d}|\alpha(x)h_i(s,x)| = 0,
		\]
		and therefore, the first term converges to zero uniformly in $x \in [0,1)^d$.

		Similarly, from the inductive hypothesis, the supremum in the second term is bounded, while, from \eqref{rhobound},   the integral in the second term converges to zero uniformly in $x \in [0,1)^d$. Thus, we conclude that 
		\[
		\lim\limits_{t \to \infty } \sup\limits_{x \in [0,1)^d} q_k(t,x) = 0,
		\]
		which completes the proof of Lemma $\ref{formmk}$.
		
		\qed

		\section{Proof of \thmref{allmoments}}
		Without loss of generality, we may assume that $\bar{\bv} = 0$, which simplifies our notation.
		Recall from \eqref{kthmoment}, 	
		\begin{equation}\label{jj}
			\EXP(n^{y}(t,x)^k) = \sum_{i = 1}^{k}S(k,i)m_i^{y}(t,x).
		\end{equation}
		We will show the following two statements by induction:\\
		\\
		(i) For each $r(t) = o(t)$, there exists the limit
		\[
		\lim\limits_{t \to \infty}\frac{m_k^{y(t)}(t,x)}{g(t,y(t))^k} = f_k(x)
		\]
		uniformly in $x \in [0,1)^d$ and  $\|y(t)\| \leq r(t)$.\\
		\\
		(ii) Let $\bar{r}(t) = o(t)$ be a function satisfying $r(t) = o(\bar{r}(t))$, with $\sqrt{t}/\bar{r}(t)\to 0$. Then 
		\[
		\lim\limits_{t \to \infty}\frac{m_k^{\bar{y}(t)}(t,x)}{g(t,y(t))^k} = 0
		\]
		uniformly in $x \in [0,1)^d$, $\|y(t)\| \leq r(t)$,  and   $\|\bar{y}(t)\| \geq \bar{r}(t)$.

		The theorem will then immediately follow from (i) since $g(t,y) \to \infty$  as $t \to \infty$ for $\|y\|\leq r(t)$ and therefore, the term with $i = k$ dominates in the sum in formula $\eqref{jj}$.
		
		For $k = 1$, using the asymptotic formula for $\rho_1(t,x,y)$ that was given in Theorem \ref{thmmain1}, we get
		\[
		m_1^{y}(t,x) = \int_{\TOR_y^d} \rho_1(t,x,z)dz = 
		\]
		\begin{equation}\label{formm1la}
			=(\sqrt{2\pi t})^{-d}\ph(x)
			\left(\int_{\TOR_y^d}\det[ D^2 \Phi(\frac{z-x}{t}) ]^{1/2}e^{-t\Phi(\frac{z-x}{t})} \ph^*(z)\,dz\right)
			\left[1 +o_L(1)\right],
		\end{equation}
		for all $x,y \in \real^d$ with $\|x-y\| \leq Lt$. Observe that the following limits exit uniformly in $x \in [0,1)^d$, $z \in \TOR_y^d$,  and $\|y(t)\| \leq r(t)$,\[
		\lim\limits_{t \to \infty}\det[ D^2 \Phi(\frac{z-x}{t})] = \det[ D^2 \Phi(0)],~~~\lim\limits_{t \to \infty} e^{-t\Phi(\frac{z-x}{t}) +t\Phi(\frac{y}{t})} = 1,
		\] 
		while $\int_{\TOR_y^d}\ph^*(z)d\,z = 1$.  Therefore, (i) holds for $k=1$. To prove (ii) for $k=1$, it enough to show that
		\begin{equation}\label{three}
			\lim\limits_{t \to \infty} e^{t\Phi(\frac{y(t)}{t}) -t\Phi(\frac{\bar{y}(t)}{t})} = 0
		\end{equation}
		uniformly in $\|y(t)\|\leq r(t)$, $\|\bar{y}(t) \| \geq \bar{r}(t)$.

		First observe that, given a small $c>0$, there exists a constant $m>0$ such that $-\mu - \Phi(v) \leq -c$ for all $\|v\| \geq m$. In addition, since $\frac{y(t)}{t} \to 0$ as $t \to \infty$, there exists $T_1>0$ such that $|\Phi\Big(\frac{y(t)}{t}\Big) + \mu| \leq c/2$ for all $t \geq T_1$.  Therefore, whenever $\|\frac{\bar{y}(t)}{t}\| \geq m$, we have $\Phi\Big(\frac{y(t)}{t}\Big) - \Phi\Big(\frac{\bar{y}(t)}{t}\Big)\leq -c/2$. for all $t \geq T_1$.
		That is, if $\|\frac{\bar{y}(t)}{t}\| \geq m$, then
		\begin{equation}\label{one}
			e^{t\Phi(\frac{y(t)}{t}) -t\Phi(\frac{\bar{y}(t)}{t})} \leq e^{-tc/2}\end{equation}
		for all $t \geq T_1$.
		
		We choose $T_2>T_1$ such that, for all $t\geq T_2$, $\|\frac{y(t)}{t}\| \leq m$. Observe that there exist constants $c_1,c_2>0$ such that \be 
		\begin{equation}\label{second}
			c_1\|
			\bv\|^2 \leq |\langle D^2\Phi\left(v\right)\bv, \bv\rangle| \leq c_2\|\bv\|^2\end{equation} for all $\bv\in \real^d$ with $\|\bv\|\leq m$.
		Whenever $\|\frac{\bar{y}(t)}{t}\| \leq m$, using Taylor's formula, for all $t \geq T_2$, there exist $\alpha_1, \alpha_2 \in (0,1)$ such that 
		\[
		t\Big(\Phi\big(\frac{y(t)}{t}\big)-\Phi\big(\frac{\bar{y}(t)}{t}\big) \Big) =  t\Big[ \langle \frac{y(t)}{t},D^2\Phi\left(\frac{\alpha_1y(t)}{t} \right)\frac{y(t)}{t} \rangle - \langle \frac{\bar{y}(t)}{t},D^2\Phi\left(\frac{\alpha_2
			\bar{y}(t)}{t} \right)\frac{\bar{y}(t)}{t} \rangle\Big]\]
		\begin{equation}\label{two}
			\leq c_2\frac{\|\bar{r}(t)\|^2}{t}\Big[\Big\|\frac{r(t)}{\bar{r}(t)}\Big\|^2- \frac{c_1}{c_2}\Big].
		\end{equation}
		Since $\sqrt{t}/\bar{r}(t) \to 0$, and $r(t) = o(\bar{r}(t))$, \eqref{two} and \eqref{one} imply \eqref{three}.
		This concludes the proof of (i) and (ii) for $k=1$. 
		
		Now, let us assume that (i) and (ii) hold up to $k-1$, where $k \geq 2$. We first prove (i) for $k$. Let $\bar{r}(t) = o(t)$ be a function satisfying $r(t) = o(\bar{r}(t))$, with $\sqrt{t}/\bar{r}(t)\to 0$.
		Recall from \eqref{exactmk},  
		\begin{equation*}
			m_k^y(t,x) =\int_0^t \int_{\real^d} \alpha(z) \sum_{i = 1}^{k-1}\beta_i^km_i^y(s,z)m_{k-i}^y(s,z) \rho_1(t-s,x,z)dz ds.
		\end{equation*}
		Let $\eps\in (0,1)$, to be selected later. Let us define the following
		\[
		A_k(t,x,y(t)) := \int_0^{\eps t} \int_{\real^d}\alpha(z) \sum_{i = 1}^{k-1}\beta_i^km_i^{y(t)}(s,z)m_{k-i}^{y(t)}(s,z)\rho_1(t-s,x,z)dz ds,
		\] 
		\[
		B_k(t,x,y(t)) : =  \int_{\eps t}^{ t} \int_{\|z-y(t)\|\geq \bar{r}(t)}\alpha(z) \sum_{i = 1}^{k-1}\beta_i^km_i^{y(t)}(s,z)m_{k-i}^{y(t)}(s,z) \rho_1(t-s,x,z)dz ds,
		\]
		\[
		C_k(t,x,y(t)) :=  \int_{\eps t}^{t} \int_{\|z-y(t)\|\leq  \bar{r}(t)}\alpha(z) \sum_{i = 1}^{k-1}\beta_i^km_i^{y(t)}(s,z)m_{k-i}^{y(t)}(s,z)\rho_1(t-s,x,z)dz ds.
		\]
		By \eqref{aronsonk}, we can choose $\eps>0$ small enough so that, for each $1\leq i \leq k-1$,\[m_i^{y}(s,z)m_{k-i}^{y}(s,z) \leq ce^{\mu t/2}\] for all $0 \leq s \leq \eps t$ and $z,y\in \real^d$.  For this fixed $\eps>0$, choosing a sufficiently large $L>0$, we use the asymptotic formula for $\rho_1(t,x,z)$ that was given in Theorem \ref{thmmain1} in the region $\|z-x\|\leq Lt$ and the estimate \eqref{rho} elsewhere, to obtain that $A_k(t,x,y) \leq c_1 e^{3\mu t/2}$, for all $x,y \in \real^d$. Therefore, there exists a constant $C>0$ such that,
		\[
		\lim\limits_{t \to \infty} \frac{A_k(t,x, y(t))}{g(t,y(t))^k} \leq \lim\limits_{t \to \infty}C (2\pi t)^{d/2}e^{(3/2 - k)\mu t} =  0,
		\]
		uniformly in $x \in [0,1)^d$ and  $\|y(t)\| \leq r(t)$, since $k\geq 2$. Next we show that
		\[
		\lim\limits_{t \to \infty} \frac{B_k(t,x, y(t))}{g(t,y(t))^k}  = 0.
		\] Since the operator $\cL$ is periodic, we first observe that $m_k^{y}(t,x) =  m_k^{y-[x]}(t,\{x\}) $ for all $k \in \naturals$, $x,y \in \real^d$,  and $t \geq 0$. For all $1 \leq i\leq k-1$, from (ii) we have 
		\[
		\lim\limits_{t \to \infty} \frac{m_i^{y}(s(t),
			z)}{g(s(t),y(t))^i} = \lim\limits_{t \to \infty} \frac{m_i^{y-[z]}(s(t),
			\{z\})}{g(s(t),y(t))^i}  =0
		\] uniformly  in $ s(t) \in (\eps t, t)$, $\|y(t)\| \leq r(t)$ and $\|z-y(t)\|\geq \bar{r}(t)$ where $[\cdot]$ denotes the greatest integer function in $d$ dimensions, and  $\{z\} = z - [z]$ . 
		Thus, it is enough to show that there exists a $C>0$ such that 
		\[
		\lim\limits_{t \to \infty} \frac{\sup\big\{g(s ,y(t))^k\big|s \in (\eps t, t)\}}{g(t,y(t))^k} \int_{\eps t}^{t}\int_{\|z-y(t)\|\geq \bar{r}(t)} \alpha(z)\rho_1(t-s,x,z)dz ds \leq C.
		\] 
		Choosing a sufficiently large $L>0$, we use the asymptotic formula for $\rho_1(t,x,z)$ that was given in Theorem \ref{thmmain1} in the region $\|z-x\|\leq Lt$ and the estimate \eqref{rho} elsewhere, to obtain that \begin{equation}\label{boundrho}
			\int_{\real^d} \rho_1(t-s,x,z)dz \leq a e^{\mu(t-s)}.
		\end{equation}
		Thus, it is enough to show that 
		\begin{equation}\label{now}
			\lim\limits_{t \to \infty} \sup\big\{e^{kt\Big[\Phi\big(\frac{y(t)}{t}\big)-\frac{s}{t}\Phi\big(\frac{y(t)}{s}\big)\Big]} e^{\mu(t-s)}\big|s \in (\eps t, t)\big\} = 1.
		\end{equation} 
		Note that  $e^{kt\Big[\Phi\big(\frac{y(t)}{t}\big)-\frac{s}{t}\Phi\big(\frac{y(t)}{s}\big)\Big]} e^{\mu(t-s)} = 1$ when $s = t$. 
		We show that, for sufficiently large $t$, the supremum in the above expression is achieved when $s = t$, when $t$ is large enough. To show the claim, for $s =  t-\delta$, we will show that 
		\[
		kt\Big[\Phi\big(\frac{y(t)}{t}\big)-\Phi\big(\frac{y(t)}{t-\delta}\big)\Big] + k\delta\Phi\big(\frac{y(t)}{t}\big) +\mu \delta<0.
		\]
		Recall that $\Phi$ is continuous and the minimum value of the function $\Phi$ is achieved at $0$, which is $\Phi(0) = -\mu<0$. In addition, recall that $r(t) = o(t)$. Thus, since $k \geq 2$, we conclude that there exists $\eta>0$ such that, for all sufficiently large $t$,\[
		k\delta\Phi\big(\frac{y(t)}{t}\big) +\mu \delta< -\eta.
		\]
		Thus, it is enough to show that, for all sufficiently large $t$, 
		\begin{equation}\label{pp}
			\Big|kt\Big(\Phi\big(\frac{y(t)}{t}\big)-\Phi\big(\frac{y(t)}{t-\delta}\big)\Big)\Big|<\eta/2.
		\end{equation}
		Indeed, for large $t$, the value of $\|y(t)/t\|$ is close to $0$, while $\|y(t)/(t-\delta) \| =\|y(t)/s\| \leq \frac{1}{\eps}\|y(t)/t\|$ is also close to $0$. Thus, using the fact that $\nabla\Phi(0) = 0$ and Taylor's formula, we obtain that \eqref{pp} holds. Thus, we have shown that the supremum in \eqref{now} is achieved when $s = t$, when $t$ is large enough. This completes the proof of \eqref{now}. Next we show that
		\[
		\lim\limits_{t \to \infty} \frac{C_k(t,x, y(t))}{g(t,y(t))^k}  =f_k(x).
		\]
		In the region $\|z-y(t)\|\leq  \bar{r}(t)$, by the inductive assumption, we can replace 
		\[\sum_{i = 1}^{k-1}\beta_i^km_i^{y(t)}(s,z)m_{k-i}^{y(t)}(s,z)~~~\text{by}~~~g(s,y(t) -[z])^k\sum_{i = 1}^{k-1}\beta_i^kf_i(z)f_{k-i}(z)\] in the integral. Therefore, we obtain 
		\[
		\lim\limits_{t \to \infty}\frac{ C_k(t,x,y(t))}{ g(t,y(t))^k}= \]\[ = \lim\limits_{t \to \infty}\sum_{i = 1}^{k-1}\beta_i^k\int_{\eps t}^{ t}   \int_{\|z-y(t)\|\leq \bar{r}(t)}\Big(\frac{g(s, y(t)-[z])}{g(t,y(t))}\Big)^k\alpha(z)f_i(z)f_{k-i}(z) \rho_1(t-s,x,z)dz ds. \] 
		Therefore, using a change of variable, it remains to show that, for each $1 \leq i \leq k-1$, 
		\begin{equation}\label{quotient}
			\lim\limits_{t \to \infty}\int_{0}^{( 1- \eps) t}  \int_{\|z-y(t)\|\leq \bar{r}(t)}\Big( \frac{g(t-s, y(t)-[z])}{g(t,y(t))}\Big)^k \alpha(z)f_i(z)f_{k-i}(z) \rho_1(s,x,z)dz ds=\end{equation}
		\[ =  \int_0^\infty\int_{[0,1)^d}e^{-k\mu s}\alpha(z)f_i(z)f_{k-i}(z)\rh(s,x,z)dzds .
		\]
		Let $\eta>0$ be fixed. From (ii), if $\|y(t)-[z]\|/\|y(t)\| \to \infty$, then,\[\Big( \frac{g(t, y(t)-[z])}{g(t,y(t))}\Big)^k \to \infty.\]
		Note that,
		\[
		\frac{g(t-s, y(t) - [z])}{g(t,y(t))} =\]\[= \Big(\frac{t}{t-s}\Big)^{(d/2)}\exp\Big[t\Big(\Phi\Big(\frac{y(t)}{t}\Big) -\Phi\Big(\frac{y(t)-[z]}{t-s}\Big)\Big) + s\Phi\Big(\frac{y(t)-[z]}{t-s}\Big)\Big].
		\]
		The term $(t/t-s)^{(d/2)}$ is bounded when $0 \leq s \leq (1-\eps)t$. Given $\delta>0$ small, using the fact that $\|z-y(t)\|\leq \bar{r}(t)= o(t), \|y(t)\|\leq r(t) = o(t)$ and $\nabla\Phi(0)= 0$, from Taylor's formula, there exists $\alpha\in (0,1)$ and $\ell(t) = \frac{y(t)-[z]}{t-s} + \alpha\Big(\frac{y(t)}{t} - \frac{y(t)-[z]}{t-s}\Big)$ such that, for all sufficiently large $t$, for all $0 \leq s\leq (1-\eps)t$,
		\[
		t\Big(\Phi\Big(\frac{y(t)}{t}\Big) -\Phi\Big(\frac{y(t)-[z]}{t-s}\Big)\Big) + s\Phi\Big(\frac{y(t)-[z]}{t-s}\Big)\]
		\[
		\leq t\|\nabla\Phi\Big(\ell(t)\Big)\| \Big\|\frac{t[z]-sy(t)}{t(t-s)}\Big\| + (-\mu + \delta )s \]
		\begin{equation}\label{ff}
			\leq  \delta \|z\| + (-\mu + 2\delta )s .\end{equation}
		Therefore, using \eqref{rho}, we conclude that
		\[
		\int_{\|z-y(t)\|\leq \bar{r}(t)}\Big( \frac{g(t-s, y(t)-[z])}{g(t,y(t))}\Big)^k \alpha(z)f_i(z)f_{k-i}(z) \rho_1(s,x,z)dz\leq \]
		\[\leq C\int_{\|z-y(t)\|\leq \bar{r}(t)} \exp\Big[k(-\mu + 2\delta )s + k\delta \|z\|\Big] \rho_1(s,x,z)dz\leq \]
		\[\leq \tilde{C}s^{-d/2}\exp\Big[k(-\mu + 2\delta )s + \mu s \Big]\int_{\|z-y(t)\|\leq \bar{r}(t)}  \exp \Big[k\delta \|z\| - \frac{\|z\|^2}{cs}\Big]dz  \leq 
		\]
		\[\leq \tilde{C}s^{-d/2}\exp\Big[-(k-1)\mu s +k2\delta s + cs\delta^2k^2/4\Big]\int_{\real^d}  \exp \Big[-\Big(\frac{\|z\|}{\sqrt{cs}} - \frac{k\delta \sqrt{cs}}{2}\Big)^2\Big]dz  .
		\]
		Now, by choosing $\delta$ small enough so that $k\delta<\mu/8$, and $c\delta^2k^2< \mu$, we have, for all $k \geq 2$,  
		\[
		-(k-1)\mu s +k\delta s + cs\delta^2k^2/4 \leq -\mu s/2.
		\]
		Therefore, there exists $m>0$ such that, for all $t$ sufficiently large, 
		
		\begin{equation}\label{some}
			\int_{m}^{( 1- \eps) t}  \int_{\|z-y(t)\|\leq \bar{r}(t)}\Big( \frac{g(t-s, y(t)-[z])}{g(t,y(t))}\Big)^k \alpha(z)f_i(z)f_{k-i}(z) \rho_1(s,x,z)dz ds< \eta/10
		\end{equation}
		uniformly in $x \in [0,1)^d$ and  $\|y(t)\| \leq r(t)$. From \eqref{rhobound}, we can show that
		\begin{equation}\label{som}
			\int_m^\infty\int_{[0,1)^d}e^{-k\mu s}\alpha(z)f_i(z)f_{k-i}(z)\rh(s,x,z)dzds  < \eta
		\end{equation}
		uniformly in $x \in [0,1)^d$. Now it remains to show that, for each $0 \leq s \leq m$, 
		\[
		\int_0^m\int_{\|z-y(t)\|\leq \bar{r}(t)}\Big( \frac{g(t-s, y(t)-[z])e^{\mu s}}{g(t,y(t))}\Big)^k \alpha(z)f_i(z)f_{k-i}(z) \rho_1(s,x,z)dz ds\]\[\to 
		\int_0^m\int_{[0,1)^d}\alpha(z)f_i(z)f_{k-i}(z)\rh(s,x,z)dzds.
		\]
		Observe that $\rho_1$ is the fundamental solution of the operator $\cL$ on $\real^d$, while $\rh$ is the fundamental solution of the same operator on  $\TOR^d$. Therefore, for each $s \geq 0$, $x \in [0,1)^d$, and each continuous $\bbZ^d-$ periodic function $h :\real^d \to \real$, we have the relation
		\[
		\int_{\real^d}h(z)\rho_1(s,x,z)dz = \int_{[0,1)^d}h(z)\rh(s,x,z)dz.
		\]
		Also, if $\|z - y(t)\| \geq \bar{r}(t)$, and $\|y\| \leq r(t)$ where $r(t) = o(\bar{r}(t))$, we conclude that, for sufficiently large $t$, $\|z-x\|\leq c\bar{r}(t)$, for some $c>0$. Therefore, for $0\leq s \leq m$, $1 \leq i \leq k-1$, for sufficiently large $t$, using \eqref{rho}, we obtain
		\begin{equation}\label{so}
			\int_0^m \int_{|z-x|\geq c\bar{r}(t)}\alpha(z)f_i(z)f_{k-i}(z)\rho_1(s,x,z)dzds <\eta.
		\end{equation}
		Thus, it remains to show that,  $1 \leq i \leq k-1$, 
		\[
		\int_{\|z-y(t)\|\leq \bar{r}(t)}\Big|\Big( \frac{g(t-s, y(t)-[z])e^{\mu s}}{g(t,y(t))}\Big)^k -1\Big|\alpha(z)f_i(z)f_{k-i}(z) \rho_1(s,x,z)dz \to 0
		\]
		uniformly in $s \in [0,m]$. Let us first prove that there exists $R>0$ such that,  for $0\leq s \leq m$, $1 \leq i \leq k-1$, 
		\begin{equation}\label{somet}
			\int_{\stackunder{\|z\|\geq R}{\|z-y(t)\|\leq \bar{r}(t)}}\Big|\Big( \frac{g(t-s, y(t)-[z])e^{\mu s}}{g(t,y(t))}\Big)^k -1\Big|\alpha(z)f_i(z)f_{k-i}(z) \rho_1(s,x,z)dz < \eta.
		\end{equation}
		As in \eqref{ff}, and using \eqref{rho}, given $\delta>0$ small, for all sufficiently large $t$, and for $0 \leq s \leq m$,  we get,
		\[
		\int_{\stackunder{\|z\|\geq R}{\|z-y(t)\|\leq \bar{r}(t)}}\Big( \frac{g(t-s, y(t)-[z])e^{\mu s}}{g(t,y(t))}\Big)^k \alpha(z)f_i(z)f_{k-i}(z) \rho_1(s,x,z)dz\leq \]
		\[\leq \tilde{C}s^{-d/2}\exp\Big[\mu s +2k\delta s + cs\delta^2k^2/4\Big]\int_{\|z\|\geq R} \exp \Big[-\Big(\frac{\|z\|}{\sqrt{cs}} - \frac{k\delta \sqrt{cs}}{2}\Big)^2\Big]dz  .
		\]
		By choosing $R>0$ large enough, the right can be made arbitrarily small uniformly for all  $0 \leq s \leq m$. Thus, \eqref{somet} holds. 
		Now it remains to show that for this positive constant $R>0$, we have 
		\begin{equation}\label{nuy}
			\lim\limits_{t \to \infty} \sup\Big\{ \big|\frac{g(t-s, y(t) - [z])e^{\mu s}}{g(t,y(t))}-1\big|\Big|s \in [0, m], \|y(t)\| \leq r(t), \|z\| \leq  R\Big\} = 0.
		\end{equation}
		To see this, as before, we observe that,
		\[
		\frac{g(t-s, y(t) - [z])e^{\mu s}}{g(t,y(t))} =  \Big(\frac{t}{t-s}\Big)^{(d/2)}e^{\Big[t\Big(\Phi\Big(\frac{y(t)}{t}\Big) -\Phi\Big(\frac{y(t)-[z]}{t-s}\Big)\Big) + s\Big(\mu +\Phi\Big(\frac{y(t)-[z]}{t-s}\Big)\Big)\Big]}.
		\]
		Given $\delta>0$ small, for each sufficiently large $t$,
		\[t\Big|\Phi\Big(\frac{y(t)}{t}\Big) -\Phi\Big(\frac{y(t)-[z]}{t-s}\Big)\Big|< \delta/2 ,~~\text{and}~~ s\Big|\mu +\Phi\Big(\frac{y(t)-[z]}{t-s}\Big)\Big|<\delta/2,
		\]
		for all $\|z\|\leq R$, $\|y(t)\|\leq r(t)$ $\|y(t)-z\|\leq \bar{r}(t)$ and $0 \leq s\leq m$. Therefore, we can choose $\delta>0$  small enough, such that, for all sufficiently large $t$, 
		\begin{equation}\label{somethi}
			\int_{\stackunder{\|z\|\leq R}{\|z-y(t)\|\leq \bar{r}(t)}}\Big|\Big( \frac{g(t-s, y(t)-[z])e^{\mu s}}{g(t,y(t))}\Big)^k -1\Big|\alpha(z)f_i(z)f_{k-i}(z) \rho_1(s,x,z)dz < \eta.
		\end{equation}
		Since $\eta>0$ was arbitrary,   \eqref{some}, \eqref{som}, \eqref{so}, \eqref{somet} and \eqref{somethi} complete the proof of (i) for $k$.

		We now prove (ii) for $k$. For fixed $r(t)$ and $\bar{r}(t)$ as in (ii), choose  $p(t)$ such that $r(t)\ll p(t) \ll  \bar{r}(t)$. Again, divide the integral in the definition of $m_k^{\bar{y}(t)}(t,x)$ into the following three integrals:
		\[
		\bar{A}_k(t,x,\bar{y}(t)) := \int_0^{\eps t} \int_{\real^d}\alpha(z) \sum_{i = 1}^{k-1}\beta_i^km_i^{\bar{y}(t)}(s,z)m_{k-i}^{\bar{y}(t)}(s,z)\rho_1(t-s,x,z)dz ds,
		\] 
		\[
		\bar{B}_k(t,x,\bar{y}(t)) : =  \int_{\eps t}^{ t} \int_{|z-\bar{y}(t)|\geq p(t)}\alpha(z) \sum_{i = 1}^{k-1}\beta_i^km_i^{\bar{y}(t)}(s,z)m_{k-i}^{\bar{y}(t)}(s,z) \rho_1(t-s,x,z)dz ds,
		\]
		\[
		\bar{C}_k(t,x,\bar{y}(t)) :=  \int_{\eps t}^{t} \int_{|z-\bar{y}(t)|\leq  p(t)}\alpha(z) \sum_{i = 1}^{k-1}\beta_i^km_i^{\bar{y}(t)}(s,z)m_{k-i}^{\bar{y}(t)}(s,z)\rho_1(t-s,x,z)dz ds.
		\] 
		From the proof of (i), following the arguments used to show that\[ A_k(t,x,y(t))/g(t,y(t))^k \to 0 ~~\text{ and } ~~B_k(t,x,y(t))/g(t,y(t))^k \to 0,\] we can also show that \[\bar{A}_k(t,x,\bar{y}(t))/g(t,y(t))^k \to 0 ~~\text{ and } ~~\bar{B}_k(t,x,\bar{y}(t))/g(t,y(t))^k \to 0,\] uniformly in $\|\bar{y}(t)\| \geq \bar{r}(t), \|y(t)\| \leq r(t)$.
		Next  we show that, for $ 1\leq i \leq k-1$,  the following limit holds uniformly in $\|\bar{y}(t)\| \geq \bar{r}(t), \|y(t)\| \leq r(t)$
		\[\lim\limits_{t \to \infty}\int_{0}^{( 1- \eps) t}  \int_{\|z-\bar{y}(t)\|\leq \bar{p}(t)}\Big( \frac{g(t-s, \bar{y}(t)-[z])}{g(t,y(t))}\Big)^k \alpha(z)f_i(z)f_{k-i}(z) \rho_1(s,x,z)dz ds= 0.
		\]
		Following the same arguments that are detailed before \eqref{some}, it is enough to show that, for all $ 1\leq i \leq k-1$, 
		\[
		\lim\limits_{t \to \infty}\int_0^m\int_{\|z-\bar{y}(t)\|\leq \bar{p}(t)}\Big( \frac{g(t-s, \bar{y}(t)-[z])}{g(t,y(t))}\Big)^k \alpha(z)f_i(z)f_{k-i}(z) \rho_1(s,x,z)dz  ds = 0.
		\]
		uniformly in $\|\bar{y}(t)\| \geq \bar{r}(t), \|y(t)\| \leq r(t)$. The idea here is that, $\|\bar{y}\|$ as well as $\|y(t)\|$ can be bounded from above by $2\|z\|$ on the domain of integration. Therefore, repeating the arguments from \eqref{ff}, using Taylor's formula, given $\delta>0$ small. since $\|z-\bar{y}(t)\|\leq {p}(t)= o(t), \|y(t)\|\leq r(t) = o(t)$ and $\nabla\Phi(0)= 0$, along with the estimate \eqref{rho}, there exists $C>0$ such that, for all sufficiently large $t$ and all $1 \leq i \leq k-1$, 
		\[
		\int_0^m \int_{\|z-\bar{y}(t)\|\leq \bar{p}(t)}\Big( \frac{g(t-s, \bar{y}(t)-[z])}{g(t,y(t))}\Big)^k \alpha(z)f_i(z)f_{k-i}(z) \rho_1(s,x,z)dz ds \]\[\leq C\int_0^m e^{-(k-1)\mu s + 2\delta s}\frac{1}{s^{d/2}}\int_{\|z-\bar{y}(t)\|\leq  \bar{p}(t)}e^{\delta \|z\| - \frac{\|z-x\|^2}{cs}} dzds.
		\]
		Since $\|\bar{y}(t) \|\geq\bar{r}(t)$ and $\|z-\bar{y}(t)\|\leq  \bar{p}(t) = o(\bar{r}(t))$, we know, for sufficiently large $t$, $\|z-x\|\geq \bar{r}(t)/2$.  Thus, there exists $a>0$ such that, for sufficiently large $t$, 
		\[
		\delta \|z\| - \frac{\|z-x\|^2}{cs} \leq \delta \|x\| +\Big( \delta \|z-x\|- \frac{\|z-x\|^2}{cs}\Big) \leq  \delta \|x\| - \frac{\|z-x\|^2}{as}. 
		\]
		Therefore, there exists a constant $\tilde{C}>0$ such that, for all $1 \leq i \leq k-1$, 
		\[\lim\limits_{t \to \infty}\int_0^m \int_{\|z-\bar{y}(t)\|\leq {p}(t)}\Big( \frac{g(t-s, \bar{y}(t)-[z])}{g(t,y(t))}\Big)^k \alpha(z)f_i(z)f_{k-i}(z) \rho_1(s,x,z)dz ds\leq \]
		\[\leq \tilde{C}\lim\limits_{t \to \infty}\int_0^m e^{-(k-1)\mu s + \delta s}\Big(\frac{1}{s^{d/2}}\int_{\|z-x\|\geq \bar{r}(t)}e^{- \frac{\|z-x\|^2}{as}} dz \Big)ds  = 0,
		\] 
		if $\delta$ is sufficiently small. 
		This concludes the proof of (ii) for $k$.  
  \bibliographystyle{abbrev}

		\bibliography{heatkernelrefs}
		
	\end{document}